\documentclass[11pt]{article}
\usepackage[includehead, includefoot, heightrounded, left=0.8in, top=1.6cm, bottom=2.4cm, right=0.8in]{geometry}
\usepackage[svgnames, dvipsnames, usenames]{xcolor}

\usepackage{tikz}
\usepackage{tgpagella}
\usepackage{graphicx}
\usepackage{amsmath, amssymb, amsthm} 
\usepackage{mathtools} 
\usepackage{bbm} 
\newcommand{\mf}{\mathfrak}
\newcommand{\mc}{\mathcal}
\newcommand{\CC}{\mathbb{C}}

\newcommand{\NN}{\mathbb{N}}

\newcommand{\fsl}{\mathfrak{sl}}
\newcounter{r}
\newcounter{s}
\newcommand\sTableau[1]{
        \foreach \x [count = \c from 1] in {#1} {
		\foreach \y [count = \d from 1] in \x{
			\node at (\d-.5,\c-.5) {\tiny$\y$}; 
			\draw (\d,\c) to (\d,\c-1);
			{\ifnum\d=1
				\draw (0,\c) to (0,\c-1);
				\fi}
			\setcounter{r}{\d}
		}
		{\ifnum\c=1
			\draw (0,0)--(\value{r},0);
			\fi}
		\draw(0,\c) to (\value{r},\c);
		\setcounter{s}{\c}}}
\newcommand{\sTBL}[1]{
\begin{matrix}
\begin{tikzpicture}[scale=.25, yscale=-1] 
	\sTableau{#1}\node at (.5,0){};
\end{tikzpicture}
\end{matrix}
}

\numberwithin{equation}{section}

\usepackage[pdftitle={Schur--Weyl duality for diagonalizing a Markov chain on the hypercube}, pdfauthor={Persi Diaconis, Andrew Lin, and Arun Ram},
colorlinks=true,linkcolor=NavyBlue,urlcolor=RoyalBlue,citecolor=PineGreen,bookmarks=true,bookmarksopen=true,bookmarksopenlevel=2,unicode=true,linktocpage]{hyperref}
\usepackage[capitalize, nameinlink]{cleveref}

\usepackage[nottoc,notlot,notlof]{tocbibind}

\newtheorem{theorem}{Theorem}[section]
\newtheorem{corollary}[theorem]{Corollary}
\newtheorem{conjecture}[theorem]{Conjecture}

\newtheorem{proposition}[theorem]{Proposition}
\newtheorem{lemma}[theorem]{Lemma}

\newtheorem*{remark}{Remark}

\title{Schur--Weyl duality for diagonalizing a Markov chain on the hypercube}
\author{
\begin{tabular}{c} Persi Diaconis\\ \small Department of Statistics \\ [-3pt] \small Department of Mathematics \\ [-3pt]\small Stanford University\end{tabular}
\begin{tabular}{c} Andrew Lin\\ \small Department of Mathematics \\ [-3pt]\small Stanford University\end{tabular}
\begin{tabular}{c} Arun Ram\\ \small Department of Mathematics \\ [-3pt]\small University of Melbourne\end{tabular}
}
\date{}

\begin{document} 

\maketitle

\begin{abstract}
We show how the tools of modern algebraic combinatorics -- representation theory, Murphy elements, and particularly Schur--Weyl duality -- can be used to give an explicit orthonormal basis of eigenfunctions for a ``curiously slowly mixing Markov chain'' on the space of binary $n$-tuples. The basis is used to give sharp rates of convergence to stationarity.
\end{abstract}

\tableofcontents

\section{Introduction}

In \cite{diaconislinram1}, we studied a ``curiously slowly mixing Markov chain'' on $C_2^n = \{(x_1, \ldots, x_n): x_i \in \{0, 1\}\}$. Started at the all-zeros state, this chain mixes in a \textit{bounded} number of steps, no matter how large $n$ is, in both the usual $\ell^1$ and $\ell^2$ distances. From more general starts, order $\log n$ steps always suffice for $\ell^1$, while order $\frac{n}{\log n}$ steps are required for $\ell^2$ from most starts. Our initial analysis was hampered by not being able to orthogonally diagonalize the chain. Recently, we succeeded in using tools of algebraic combinatorics to find an orthonormal basis of eigenvectors. This lets us do more refined analysis.

The Markov chain in question is the ``Burnside process,'' a broadly useful procedure for sampling a uniform orbit of a finite group acting on a finite set. This chain has resisted analysis, even in very special cases, and the present step forward seems like a breakthrough.

Let $\mf{X}$ be a finite set and $G$ a finite group acting on $\mf{X}$. The action splits $\mf{X}$ into disjoint orbits 
\[
    \mf{X} = \mc{O}_1 \cup \mc{O}_2 \cup \cdots \cup \mc{O}_Z.
\]
The \textit{Burnside process} gives a way of choosing a uniformly distributed orbit (in contrast to a uniformly distributed element of $\mf{X}$), which is a basic task for P\'olya's problems of enumeration under symmetry. It consists of two steps:
\begin{itemize}
    \item From $x \in \mf{X}$, choose $s$ uniformly from the set $G_x = \{s: x^s = x\}$.
    \item From $s \in G$, choose $y$ uniformly from the set $\mf{X}_s = \{y: y^s = y\}$.
\end{itemize}

The chance of moving from $x$ to $y$ in one step of the chain is 
\[
    K(x, y) = \frac{1}{|G_x|} \sum_{s \in G_x \cap G_y} \frac{1}{|\mf{X}_s|}.
\]
This is an ergodic, reversible Markov chain on $\mf{X}$ with stationary distribution
\[
    \pi(x) = \frac{1}{Z|\mc{O}_x|}, 
\]
where $Z$ is the total number of orbits and $\mc{O}_x$ is the orbit containing $x$. Thus, running the chain and simply recording the current orbit gives a Markov chain on orbits with a uniform stationary distribution. 

Examples on huge state spaces in various applications (see for instance \cite{polyatrees}, \cite{contingency}, \cite{burnsideimportance}) empirically indicate extremely rapid mixing, but it appears difficult to prove this in any of these instances. Instead, a ``simplest example'' previously studied in the sequence of papers \cite{jerrum}, \cite{boseeinstein}, \cite{diaconiszhong}, \cite{diaconislinram1} is instructive. Let the set $\mf{X} = C_2^n$ be the binary $n$-tuples, and let the symmetric group $G = S_n$ act by permuting coordinates. Thus, the orbit decomposition is 
\[
    \mf{X} = \mc{O}_0 \cup \mc{O}_1 \cup \cdots \cup \mc{O}_n, \quad \text{where} \quad \mc{O}_i = \{x \in C_2^n: |x| = i\},
\]
and where $|x|$ denotes the number of ones in $x$. The two steps of the Burnside process can be explicitly described in this case:
\begin{itemize}
    \item From $x \in C_2^n$, choose $s \in S_n$ with $x^s = x$ uniformly. This is easy to do, since $G_x = S_i \times S_{n-i}$ for $i = |x|$, where we permute the zeros and the ones each among themselves.
    \item From $s \in S_n$, choose $y \in C_2^n$ with $y^s = y$. This too is easy to do: simply break $s$ into its disjoint cycles, color each one $0$ or $1$ with probability $\frac{1}{2}$, and install the values according to $s$.
\end{itemize}

The resulting Markov transition matrix $K(x, y)$ has a simple description, given in \cite[Proposition 3.1]{diaconislinram1}, and its stationary distribution is $\pi(x) = \frac{1}{(n+1) \binom{n}{|x|}}$. Its eigenvalues can be described as follows (\cite[Theorem 1.2]{diaconislinram1}): we have
\begin{align}\label{unlumpedeigenvalues}
&\beta_0 = 1, \nonumber \\
&\beta_k = \frac{\binom{2k}{k}^2}{2^{4k}} \text{ with multiplicity }\binom{n}{2k} \text{ for }1 \le k \le \left\lfloor \frac{n}{2}\right \rfloor, \\
&0 \text{ with multiplicity }2^{n-1}.\nonumber
\end{align}

To describe convergence of the Markov chain, define the $\ell^1$ or \textit{total variation} distance
\[
    ||K_x^\ell - \pi||_{\text{TV}} = \frac{1}{2} \sum_{y \in C_2^n} |K^\ell(x, y) - \pi(y)|,
\]  
and the $\ell^2$ or \textit{chi-square} distance
\[
    \chi_x^2(\ell) = \sum_y \frac{(K^\ell(x, y) - \pi(y))^2}{\pi(y)} = \left|\left|\frac{K^\ell}{\pi} - 1\right|\right|^2_2.
\]
The usual route to bounding convergence is to ``bound $\ell^1$ by $\ell^2$ and then bound $\ell^2$ using eigenvalues.'' More precisely,
\begin{equation}\label{l1vsl2bound}
4||K_x^\ell - \pi||_{\text{TV}}^2 \le \chi_x^2(\ell) = \sum_{\beta_i \ne 1} f_i^2(x) \beta_i^{2\ell},
\end{equation}
where the sum runs over all nontrivial eigenvalues of $K$ and $\{f_i(x)\}$ is an orthonormal basis of eigenvectors of eigenvalues $\beta_i$.

In \cite{diaconislinram1}, we found an elegant basis of eigenvectors (which is no simple task in problems with highly degenerate eigenspaces), but, alas, these eigenvectors were \textit{not} orthogonal and so we were forced to work with the average $\ell^2$ distance
\[
    \chi_{\text{avg}}^2(\ell) = \sum_x \pi(x) \chi_x^2(\ell) = \sum_{\beta_i \ne 1} \beta_i^{2\ell}.
\]
The explicit form of the eigenvalues and multiplicities from \cref{unlumpedeigenvalues} was enough to prove sharp cutoff for $\chi_{\text{avg}}^2(\ell)$ at $\ell = \frac{\log 2}{2} \frac{n}{\log n}$. This could then be refined to show (\cite[Theorem 1.1]{diaconislinram1}) that $\chi_x^2(\ell)$ tends to zero for all $x$ if $\ell \ge (1 + \varepsilon) \frac{\log 2}{2} \frac{n}{\log n}$, and also that $\chi_x^2(\ell)$ tends to $\infty$ for ``most $x$'' after order $\frac{n}{\log n}$ steps.

In contrast, previous analysis (\cite{boseeinstein}, \cite{diaconiszhong}) had shown that the chain started at the all-zeros state mixes in a \textit{bounded} number of steps in both $\ell^1$ and $\ell^2$. On the other hand, Nestoridi and Sly \cite{nestoridisly} establish worst-case $\ell^1$ cutoff in $\frac{1}{\log 16} \log n$ steps, as well as a Gaussian limit profile at window $\sqrt{\log n}$. This leaves the natural question ``how does the $\ell^2$ mixing time get from bounded to order $\frac{n}{\log n}$?'' For example, what happens if the chain starts at the state $(0, \cdots, 0, 1)$, or similarly a state with $i$ ones with $i$ constant or slowly growing with $n$?

More detailed descriptions and extensive references to previous work may be found in the introduction to \cite{diaconislinram1}. Here, we now turn to a description of the present paper.

Set $\ell^2(\pi) = \{f: C_2^n \to \CC\}$ with the inner product $\langle f, g \rangle = \sum_x f(x) \overline{g(x)} \pi(x)$. The transition matrix $K(x, y)$ operates on $\ell^2$ as
\[
    Kf(x) = \sum_y K(x, y) f(y).
\]
By construction, $K$ is a self-adjoint operator which commutes with the action of $S_n$ on $\ell^2(\pi)$ (meaning that for all $x, y \in C_2^n$ and $s \in S_n$, we have $K(x, y) = K(x^s, y^s)$). So as a representation of $S_n$, $\ell^2(\pi)$ decomposes as a direct sum of irreducible representations of $S_n$. Schur's lemma implies that the operator $K$ sends isotypic pieces of this decomposition into themselves, and decomposing further into eigenspaces of $K$ refines this. (See \cite{repbook} -- particularly Section 7 for material on the symmetric group -- or \cite{jamessymmetric} for this classical approach.) One result of our development is a complete description of how the isotypic pieces decompose relative to the eigenvalue multiplicities (\cref{eigenvaluesplittingresult} below).

The Lie algebra $\mathfrak{sl}_2(\CC)$ also acts on $\ell^2(\pi)$ (explained in \cref{schurweylsubsection2}), commuting with the action of $S_n$. Schur--Weyl duality implies that $\mathfrak{sl}_2$ and $S_n$ are each the full centralizer of the other. One main accomplishment of the present paper is to refine these symmetry decompositions into a ``magical'' orthonormal basis of $\ell^2(\pi)$, which \textit{does not} depend on $K(x, y)$. Thus this basis is available for other Markov chains on $C_2^n$ that commute with the action of $S_n$, some of which are compiled in \cite[Section 6.1]{diaconislinram1}. We then subsequently use properties of the transition matrix to transform this basis into an orthonormal basis of eigenvectors of $K$, described in \cref{orthoevthm}. This use of Schur--Weyl duality is the content of \cref{schurweylsection} and \cref{schurweylproofsection}.

Throughout previous and current analysis, a collection of classical orthogonal polynomials play a central role. These are the Chebyshev and Hahn polynomials: the Chebyshev polynomials are the orthogonal polynomials for the uniform distribution on the integers $\{0, 1, \cdots, n\}$, and the Hahn polynomials are analogously for the ``beta-binomial'' distribution on that same space. \cref{schurweylsubsection2} explains their presence and uses classical properties to give the explicit decompositions promised above. Classical formulas also help give explicit norming constants to simplify ``impossible'' hypergeometric sums.

Next, \cref{schurweylsubsection3} applies all of this machinery to give sharp rates of convergence in $\ell^2$ for the chain started at the state $e_n = (0, \cdots, 0, 1)$. The upshot is that a bounded number of steps are necessary and sufficient, just like for the all-zeros state:

\begin{theorem}\label{swdistancefrom1}
For the binary Burnside process started from the state $e_n = (0, 0, \cdots, 1)$ (or any other state with a single $1$), a constant number of steps is necessary and sufficient in $\ell^2$ (and therefore also in $\ell^1$). More precisely, for all $n, s \ge 3$, the chi-square distance to stationarity after $s$ steps satisfies
\[
    5\left(\frac{1}{4}\right)^{2s} \le \chi^2_{e_n}(s) \le 270 \left(\frac{1}{4}\right)^{2s}.
\]
\end{theorem}

The proof also gives a template for bounds from other starting states, but we have not pushed this analysis further.

Finally, a host of other results follow from the Schur--Weyl decomposition, such as a description of $K(x, y)$ as a polynomial in the usual generators $\{e, f, h\}$ of the action of $\mathfrak{sl}_2$. Comments on this and other remarks appear in \cref{miscremarkssection}.

\section*{Acknowledgements}

We thank Evita Nestoridi, Allan Sly, Amol Aggarwal, Yuval Peres, Michael Howes, Jonathan Hermon, Lucas Teyssier, Laurent Bartholdi, Christoph Koutschan, Richard Stanley, Sourav Chatterjee, and Ferenc Balogh for helpful discussions and ideas. We also greatly appreciate an anonymous referee for a careful report with many valuable comments and corrections. 

AI tools were used to correct and refine the ideas in the proof of \cref{tableauorthogonality}, but all text in the manuscript is written fully by the authors.

P.D.\ was partially supported by the NSF under Grant No.\ DGE-1954042, and A.L.\ was partially supported by the NSF under Grant No.\ DGE-2146755.

\section{Construction of orthonormal eigenvectors}\label{schurweylsection}

Throughout this paper, in order to make formulas and proofs easier to parse, we will change notation from the usual probabilistic conventions. Specifically, we will view our function space as a tensor product $V^{\otimes n}$ over the $n$ coordinates of $C_2^n$, and we will write all functions as linear combinations of the basis vectors $v_S$ (for subsets $S \subseteq \{1, 2, \ldots, n\}$). 

In this section, we describe how to construct orthogonal subspaces indexed by Young tableaux and obtain explicit formulas for the vectors that lie in those subspaces. In particular, the initial orthogonal decomposition here does not depend on our operator $K$, so (as described in the introduction) it may be useful for other Markov chains on binary $n$-tuples that are $S_n$-invariant. We then use this orthogonal decomposition to describe a complete set of orthogonal eigenvectors $\{f_Q^{m, \ell}\}$ indexed by integers $m, \ell$ (which dictate the eigenvalue) and Young tableaux $Q$ of shape $(n-m, m)$ in \cref{orthoevthm} below. 

\medskip

We first establish all of the notation we will use in this paper. Let $V = \CC\text{-span}\{v_0, v_1\}$, so that the tensor space
\[
    V^{\otimes n} \quad \text{has basis}\quad \{ v_{i_1}\otimes \cdots \otimes v_{i_n} \,|\, i_1, \ldots, i_n \in \{0,1\}\}.
\]
For each subset $S\subseteq \{1,\ldots, n\}$, we write
\[
    v_S = v_{i_1}\otimes \cdots \otimes v_{i_n}, \quad \text{where} \quad i_\ell = \begin{cases} 1 &\hbox{if $\ell \in S$}, \\ 0 &\hbox{if $\ell\not\in S$.} \end{cases}
\]
The subspaces corresponding to the orbits $\mc{O}_\ell = \{x \in C_2^n: |x| = \ell\}$ are
\[
    V^{(\ell)} = \CC\text{-span}\left\{ v_S\ \middle|\ \vert S \vert = \ell\right\}, \quad \text{so that} \quad V^{\otimes n} = \bigoplus_{\ell=0}^n V^{(\ell)}.
\]
Our inner product on $V^{\otimes n}$ can then be written in the form
\begin{equation}\label{lperp}
\langle v_S, v_T \rangle_\pi = \frac{1}{n+1}\frac{1}{\binom{n}{\vert S\vert}} \delta_{ST},
\quad \text{where in particular } V^{(j)} \perp V^{(\ell)}\text{ if }j\ne \ell.
\end{equation}

In addition to the orthogonal subspaces $V^{(\ell)}$, we may also define a set of different orthogonal subspaces indexed by standard Young tableaux. For this, recall that $S_n$ acts by permutations on the coordinates, and so the group algebra of the symmetric group $\CC[S_n]$ acts on $V^{\otimes n}$ by 
\[
    wv_S = v_{wS} \quad \text{ for } w\in S_n \text{ and }S\subseteq \{1, \ldots, n\}.
\]
For $r \in \{2, \ldots, n\}$, the \textit{Jucys-Murphy elements} in the group algebra of the symmetric group (see \cite[Eq. (3.5)]{ramseminormal}, or the report \cite{murphy}, or Murphy's original paper \cite{murphy2}) are the mutually commuting operators
\begin{equation}\label{JMelts}
M_r = \sum_{i=1}^{r-1} s_{ir}, \quad \text{where } s_{ir} \text{ is the transposition that switches } i \text{ and }r.
\end{equation}
Let $\hat S_n^{(n-m,m)}$ denote the set of standard Young tableaux of shape $(n-m, m)$. For $Q \in \hat S_n^{(n-m,m)}$, let $Q(r)$ denote the box containing $r$ in $Q$. Define the \textit{content} of the box via
\[
    \mathrm{ct}(Q(r)) = y-x \text{ if } Q(r) \text{ is in row }x \text{ and column }y,
\]
and also define the subspace of simultaneous eigenvectors
\[
    V_Q = \{ m\in V^{\otimes n}\,|\, M_rm = \mathrm{ct}(Q(r)) m \text{ for all }r \in \{2, \ldots, n\}\}.
\]
These subspaces are mutually orthogonal as $Q$ varies for the following reason. For any transposition $w \in S_n$, we have 
\[
    \langle wv_S, wv_T\rangle_\pi = \langle v_S, v_T\rangle_\pi \implies \langle wv_S, v_T \rangle_\pi = \langle v_S, wv_T \rangle_\pi
\]
because $w^2 = 1$. Thus any real linear combination of transpositions in $\CC[S_n]$ is self-adjoint, meaning in particular that all $M_j$ are self-adjoint. Now let $P \ne Q$ be any standard tableaux with $n$ boxes; there must exist some $i \in \{1, \ldots, n\}$ such that $\mathrm{ct}(P(i))\ne \mathrm{ct}(Q(i))$. Then we have for any $p \in V_P$ and $q \in V_Q$ that
\[
    \mathrm{ct}(P(i))\langle p,q\rangle_\pi = \langle M_i p , q\rangle_\pi = \langle p, M_i q\rangle_\pi = \mathrm{ct}(Q(i))\langle p, q\rangle_\pi,
\]
meaning that $\langle p,q\rangle_\pi=0$. This means that for any Young tableaux $P, Q$,
\begin{equation}\label{Qperp}
\text{if }P \ne Q,\text{ then }V_P \perp V_Q.
\end{equation}
We thus obtain a refined decomposition of $V^{\otimes n}$ by defining, for all $\ell \in \{0,1,\ldots, n-m\}$ and $Q\in \hat{S}^{(n-m,m)}_n$, the subspace
\[
    V_Q^{(\ell)} = V_Q \cap V^{(\ell)}.
\]
Combining  \cref{lperp} and \cref{Qperp} gives
\begin{equation}\label{lQperp}
V_P^{(j)} \perp V_Q^{(\ell)} \text{ unless } P = Q \text{ and } j = \ell.
\end{equation}
By Schur--Weyl duality, as well as the representation theory of $\fsl_2$ and of the symmetric group $S_n$ (see \cite[Ex.\ 6.30 and (11.6)]{fulton} and \cite[(3.5), (3.11), and Thm. 3.14]{ramseminormal}), we have
\begin{equation}\label{Vndecomp}
\dim(V_Q^{(\ell)}) = 1 \text{ for }m \le \ell \le n-m \qquad \text{and} \qquad
V^{\otimes n} = \bigoplus_{m=0}^{\lfloor n/2\rfloor} \bigoplus_{Q\in \hat S_n^{(n-m,m)}} \bigoplus_{\ell=m}^{n-m} V^{(\ell)}_Q.
\end{equation}
Furthermore, it is a consequence of Schur--Weyl duality that for any $Q \in \hat{S}_n^{(n-m, m)}$ and any $m \le \ell \le n-m$,
\begin{align}\label{schurweylsplittingmechanism}
&\text{each } V_Q = \bigoplus_{\ell=m}^{n-m} V_Q^{(\ell)} \text{ is an irreducible }\fsl_2\text{-invariant subspace of }V^{\otimes n}, \text{ and}\nonumber \\
&\text{each } V^{(\ell)}_{(n-m,m)} = \bigoplus_{Q \in \hat{S}_n^{(n-m,m)}} V_Q^{(\ell)} \text{ is an irreducible }S_n\text{-invariant subspace of }V^{\otimes n}.
\end{align}
We will elaborate more on this decomposition in \cref{schurweylsubsection2}. 

\medskip

We now wish to describe each of the one-dimensional subspaces $V_Q^{(\ell)}$ explicitly. First, we show a useful calculation:

\begin{lemma}\label{intprops}
Take the notation in \cref{JMelts}, and also define $M_1 = 0$ and $s_j = s_{j,j+1}$. Viewing all elements of $\CC[S_n]$ as operators on $V^{\otimes n}$, define
\begin{equation}\label{taud}
\tau_j = s_j + \frac{1}{M_j-M_{j+1}} \quad \text{for  } j\in \{1, \ldots, n-1\}.
\end{equation}
(Here $\frac{1}{M_j-M_{j+1}}$ denotes the inverse of the operator $M_j - M_{j+1}$. This inverse exists because $M_j - M_{j+1}$ acts by $\mathrm{ct}(Q(j)) - \mathrm{ct}(Q(j+1))$ on each $V_Q^{(\ell)}$, and the consecutive boxes $j$ and $j+1$ must lie in different diagonals of $Q$.) Then we have the relations
\begin{equation}
\tau_j M_j = M_{j+1}\tau_j, \quad M_j\tau_j = \tau_jM_{j+1}, \quad \text{ and } \quad \tau_j^2 = \frac{(M_j-M_{j+1}+1)(M_j-M_{j+1}-1)}{(M_j-M_{j+1})^2}.
\end{equation}
\end{lemma}
\begin{proof}
We have $M_{j+1} = s_jM_js_j+s_j$ for all $j$, so that $M_{j+1}s_j = s_j M_j + 1$. Therefore, using that $M_j$ and $M_{j+1}$ commute,
\begin{align*}
M_{j+1}\tau_j &= M_{j+1}\Big( s_j+\frac{1}{M_j-M_{j+1}} \Big) \\
&= M_{j+1}s_j + \frac{M_{j+1}}{M_j-M_{j+1}} \\
&= s_jM_j +1+\frac{M_{j+1}}{M_j-M_{j+1}} \\
&= s_jM_j + \frac{M_j}{M_j-M_{j+1}} \\
&= \left(s_j + \frac{1}{M_j-M_{j+1}}\right)M_j = \tau_j M_j,
\end{align*}
proving the first equality. The second equality follows by an identical argument except with all multiplications in the reverse order. Putting those two facts together yields
\[
    (M_{j+1} - M_j)\tau_j = \tau_jM_j - M_j\tau_j = \tau_j(M_j - M_{j+1}).
\]
Therefore we have $\tau_j(\frac{1}{M_j - M_{j+1}}) = (\frac{1}{M_{j+1} - M_j})\tau_j$, which shows that
\begin{align*}
\tau_j^2 &=\tau_j \Big( s_j+\frac{1}{M_j-M_{j+1}} \Big) \\
&= \tau_j s_j + \frac{1}{M_{j+1}-M_j} \tau_j \\
&=  \Big( s_j+\frac{1}{M_j-M_{j+1}} \Big)s_j +\frac{1}{M_{j+1}-M_j}\Big( s_j+\frac{1}{M_j-M_{j+1}} \Big) \\
&=1-\frac{1}{(M_j-M_{j+1})^2}  \\
&= \frac{(M_j-M_{j+1})^2-1}{(M_j-M_{j+1})^2} \\
&= \frac{(M_j-M_{j+1}+1)(M_j-M_{j+1}-1)}{(M_j-M_{j+1})^2},
\end{align*}
completing the proof.
\end{proof}
Now if $j$ and $j+1$ are not in the same row or same column of the tableau $Q$, then $\tau_j^2$ acts as a nonzero constant because $\mathrm{ct}(Q(j)) - \mathrm{ct}(Q(j+1)) \notin \{-1, 0, 1\}$ (note that $0$ is not possible by monotonicity of the rows and columns of $Q$, meaning that $j+1$ cannot be on the same diagonal as $j$). Thus \cref{intprops} implies that
\begin{equation}\label{intiso}
\tau_j \colon V_Q^{(\ell)} \to V_{s_jQ}^{(\ell)}
\quad \text{is a vector space isomorphism if }j, j+1 \text{ are not in the same row or column of }Q.
\end{equation}

\medskip

\noindent \textit{\underline{Definition of eigenvectors}}: With this, we are now ready to define vectors $g_Q^{m,i}$ and $f_Q^{m,\ell}$ corresponding to the various one-dimensional subspaces mentioned above. Let $m \in \{0, 1, \ldots, \lfloor n/2\rfloor\}$. Define the \textit{column reading tableau} of shape $(n-m,m)$ to be the Young tableau
\[
    T = 
    \begin{array}{l}    \boxed{1}\boxed{3}\boxed{5}\boxed{\phantom{1}\cdots\phantom{1}}\boxed{2m-1}\boxed{2m+1}\boxed{2m+2}\boxed{\phantom{1}\cdots\phantom{1}}\boxed{n-1}\boxed{\phantom{1}n\phantom{1}} \\\boxed{2}\boxed{4}\boxed{6}\boxed{\phantom{1}\cdots\phantom{1}}\boxed{2m\phantom{22\,2}}
    \end{array}.
\]
Let $i, \ell \in \{0,1,\ldots, n-2m\}$, and let $\mathbb{S}(n-2m)_\ell$ be the set of subsets of $\{1, \ldots, n-2m\}$ with cardinality $\ell$. Define the scalars
\begin{equation}\label{Tscalarsdefn}
    T^{(\ell)}_{m,n}(i) = \sum_{S\in \mathbb{S}(n-2m)_\ell} (-1)^{m+\vert S\cap \{1, \ldots, i\}\vert} \binom{2m + \ell}{m+\vert S\cap \{1, \ldots, i\}\vert}.
\end{equation}
(These numbers turn out to be the values of certain orthogonal polynomials -- see \cref{actuallyhahn} in \cref{schurweylsubsection2} -- but we will not need that fact for this first proof.) Using these scalars, we first define the following vectors associated to the column reading tableau $T$:
\begin{equation}\label{fgTd}
g^{m, i}_T = (v_{01}-v_{10})^{\otimes m}\otimes \Big(\sum_{S\in \mathbb{S}(n-2m)_i} v_S\Big) \quad \text{and} \quad f^{m, \ell}_T = \sum_{i=0}^{n-2m} T_{m,n}^{(\ell)}(i) g^{m, i}_T.
\end{equation}
Here the subscript of $v_{01}$ stands for the subset $\{2\}$ of $\{1, 2\}$, and the tensor product notation stands for concatenation of the subsets (so that, for example, $v_{01} \otimes v_{01} = v_{0101} = v_{\{2, 4\}}$). In words, $g_T^{m, i}$ is a linear combination of particular $v_S$s with $|S| = m + i$, and $f_T^{m, \ell}$ takes a certain linear combination of these vectors over the various ``levels'' $i$ with coefficients coming from orthogonal polynomial scalars. These $f_T^{m, \ell}$ are the beginning of our orthogonal eigenvector basis.

\medskip

We now explain how to define the vectors associated to any other tableau $Q$ of shape $(n-m, m)$. If $Q$ is the standard Young tableau of shape $(n-m,m)$ with $\boxed{a_1}\boxed{a_2}\boxed{\phantom{a_k}\cdots\phantom{a_m}}\boxed{a_m}$ in the second row, then $a_r \ge 2r$ for all $1 \le r \le m$. Therefore, we can apply a sequence of adjacent transpositions of the boxes of $T$ to get to $Q$:
\begin{equation}\label{cQd}
    Q = c_Q^{(1)}\cdots c_Q^{(m)}T, \quad \text{where} \quad c_Q^{(r)} = 
    \begin{cases}
    s_{a_r-1}\cdots s_{2r+1}s_{2r},&\hbox{if $a_r>2r$,} \\
    1, &\hbox{if $a_r=2r$.}
    \end{cases}
\end{equation}
Thanks to \cref{intiso}, the corresponding map $\tau_Q \colon V_T^{(m+i)}\to V_Q^{(m+i)}$ defined by
\begin{equation}\label{tauQd}
\tau_Q = \tau_Q^{(1)}\cdots \tau_Q^{(m)}, \quad \text{where} \quad \tau_Q^{(r)} = 
\begin{cases}
\tau_{a_r-1}\cdots \tau_{2r+1}\tau_{2r},
&\hbox{if $a_r>2r$,} \\
1, &\hbox{if $a_r=2r$,}
\end{cases}
\end{equation}
is a vector space isomorphism. In particular, if $a_r > 2r$, then $\tau_Q^{(r)} = \big(s_{a_r-1}-\frac{1}{a_r-2r+1}\big)\cdots (s_{2r+1}-\hbox{$\frac{1}{3}$}) (s_{2r}-\hbox{$\frac12$})$ is an expression for $\tau_Q^{(r)}$ in terms of the simple reflections in $S_n$. We then define the following vectors associated to the tableau $Q$ (of shape $(n-m, m)$), for any $m \in \{0,1,\ldots, \lfloor n/2\rfloor\}$ and $\ell, i \in \{0, 1, \ldots,n-2m\}$:
\begin{equation}\label{fQdef}
g_Q^{m, i} = \tau_Q g_T^{m, i} \quad \text{and} \quad f^{m, \ell}_Q = \tau_Q f^{m, \ell}_T.
\end{equation}
We will soon see that the $g_Q^{m, i}$s are vectors that lie in various $V_Q^{(\ell)}$ subspaces, while the $f^{m, \ell}_Q$s form an (unnormalized) orthogonal basis of eigenvectors. Before we state the main result, we define some constants that appear crucially in the norms. For the Young tableau $Q$ notated above, define the scalar
\begin{equation}\label{gmQ}
\gamma_Q = \gamma_Q^{(1)}\cdots \gamma_Q^{(m)}, \quad \text{where} \quad \gamma_Q^{(r)} = 
\begin{cases}
\frac{((a_r-2r+1)^2-1)}{(a_r-2r+1)^2}\cdots \frac{(3^2-1)}{3^2}\cdot \frac{(2^2-1)}{2^2}, 
&\hbox{if $a_r>2r$,} \\
1, &\hbox{if $a_r=2r$.}
\end{cases}
\end{equation}

\begin{theorem}\label{orthoevthm}
Using notation as in \cref{fQdef} and \cref{gmQ}, with respect to the inner product $\langle \cdot, \cdot\rangle_\pi$, the set
\[
    \left\{ f^{m,\ell}_Q: m \in \{0,1,\ldots, \lfloor n/2\rfloor\}, \quad Q\in \hat{S}_n^{(n-m,m)}, \quad \ell \in \{0,1,\ldots, n-2m\} \right\}
\]
is an orthogonal basis of $V^{\otimes n}$. With the notation for eigenvalues $\beta_k$ as in \cref{unlumpedeigenvalues}, 
\[
    f^{m,\ell}_Q \text{ is an eigenvector of }K_n \text{ of eigenvalue }\begin{cases} \beta_{(m+\ell)/2} & \text{if }m + \ell \text{ is even}, \\ 0 & \text{otherwise,} \end{cases}
\]
and 
\begin{align}\label{fmlqnorm}
    \langle f^{m, \ell}_Q, f^{m, \ell}_Q\rangle_{\pi} &= \gamma_Q \langle f_T^{m,\ell}, f_T^{m,\ell} \rangle_{\pi} \nonumber \\
    &= \gamma_Q\frac{2^m}{n+1} \sum_{i=0}^{n-2m} \big(T_{m,n}^{(\ell)}(i)\big)^2 \frac{\binom{n-2m}{i}}{\binom{n}{m+i}}.
\end{align}
We also express this sum in closed form in \cref{hahnpolynomialnorms}.
\end{theorem}

To illustrate a concrete example, we write out the list of $g$ and $f$ vectors in the case $n = 3$. In \cref{examplegQvecs}, the first three columns indicate the values of $m$, $\ell$, and the tableau $Q$, and the last eight columns are the entries of the vector $g_Q^{m, \ell}$ evaluated at each state. (For example, the column for $011$ corresponds to $v_{\{2, 3\}} = v_0 \otimes v_1 \otimes v_1$.) The analogous table for the $f_Q^{m, \ell}$ vectors appears in \cref{examplefQvecs}, along with an extra column for displaying the normalizing factors $\langle f_Q^{m, \ell}, f_Q^{m, \ell} \rangle$. 

\begin{figure}
\begin{center}
\begin{tabular}{|c|c|c||c|c|c|c|c|c|c|c|}
\hline
$m$ & $\ell$ & $Q$ & $000$ & $001$ & $010$ & $011$ & $100$ & $101$ & $110$ & $111$ \\
\hline\hline
$0$ & $0$ & $\begin{array}{l}\boxed{1}\boxed{2}\boxed{3}\end{array}$ & $1$ & $0$ & $0$ & $0$ & $0$ & $0$ & $0$ & $0$ \\
\hline
$0$ & $1$ & $\begin{array}{l}\boxed{1}\boxed{2}\boxed{3}\end{array}$ & $0$ & $1$ & $1$ & $0$ & $1$ & $0$ & $0$ & $0$ \\
\hline
$1$ & $0$ & $\begin{array}{l}\boxed{1}\boxed{2} \\\boxed{3}\end{array}$ & $0$ & $1$ & $-\frac{1}{2}$ & $0$ & $-\frac{1}{2}$ & $0$ & $0$ & $0$ \\ 
\hline
$1$ & $0$ & $\begin{array}{l}\boxed{1}\boxed{3} \\\boxed{2}\end{array}$ & $0$ & $0$ & $1$ & $0$ & $-1$ & $0$ & $0$ & $0$ \\
\hline
$0$ & $2$ & $\begin{array}{l}\boxed{1}\boxed{2}\boxed{3}\end{array}$ & $0$ & $0$ & $0$ & $1$ & $0$ & $1$ & $1$ & $0$ \\
\hline
$1$ & $1$ & $\begin{array}{l}\boxed{1}\boxed{2} \\\boxed{3}\end{array}$ & $0$ & $0$ & $0$ & $\frac{1}{2}$ & $0$ & $\frac{1}{2}$ & $-1$ & $0$ \\
\hline
$1$ & $1$ & $\begin{array}{l}\boxed{1}\boxed{3} \\\boxed{2}\end{array}$ & $0$ & $0$ & $0$ & $1$ & $0$ & $-1$ & $0$ & $0$ \\
\hline
$0$ & $3$ & $\begin{array}{l}\boxed{1}\boxed{2}\boxed{3}\end{array}$ & $0$ & $0$ & $0$ & $0$ & $0$ & $0$ & $0$ & $1$ \\
\hline
\end{tabular}
\end{center}
\caption{The eight vectors $g_Q^{m, \ell}$ for $n = 3$. Observe that we have $g_Q^{m, \ell} \in V^{(m+\ell)}$ in all cases; that is, the vector $g_Q^{m, \ell}$ is only supported on the states $S$ where $|S| = m+\ell$.}\label{examplegQvecs}
\end{figure}

\begin{figure}
\begin{center}
\begin{tabular}{|c|c|c||c|c|c|c|c|c|c|c||c|}
\hline
$m$ & $\ell$ & $Q$ & $000$ & $001$ & $010$ & $011$ & $100$ & $101$ & $110$ & $111$ & $\langle f_Q^{m, \ell}, f_Q^{m, \ell} \rangle$ \\
\hline\hline
$0$ & $0$ & $\begin{array}{l}\boxed{1}\boxed{2}\boxed{3}\end{array}$ & $1$ & $1$ & $1$ & $1$ & $1$ & $1$ & $1$ & $1$ & $1$ \\
\hline
$0$ & $1$ & $\begin{array}{l}\boxed{1}\boxed{2}\boxed{3}\end{array}$ & $3$ & $1$ & $1$ & $-1$ & $1$ & $-1$ & $-1$ & $-3$ & $5$ \\
\hline
$1$ & $0$ & $\begin{array}{l}\boxed{1}\boxed{2} \\\boxed{3}\end{array}$ & $0$ & $-2$ & $1$ & $-1$ & $1$ & $-1$ & $2$ & $0$ & $1$ \\ 
\hline
$1$ & $0$ & $\begin{array}{l}\boxed{1}\boxed{3} \\\boxed{2}\end{array}$ & $0$ & $0$ & $-2$ & $-2$ & $2$ & $2$ & $0$ & $0$ & $\frac{4}{3}$ \\
\hline
$0$ & $2$ & $\begin{array}{l}\boxed{1}\boxed{2}\boxed{3}\end{array}$ & $3$ & $-3$ & $-3$ & $-3$ & $-3$ & $-3$ & $-3$ & $3$ & $9$ \\
\hline
$1$ & $1$ & $\begin{array}{l}\boxed{1}\boxed{2} \\\boxed{3}\end{array}$ & $0$ & $-3$ & $\frac{3}{2}$ & $\frac{3}{2}$ & $\frac{3}{2}$ & $\frac{3}{2}$ & $-3$ & $0$ & $\frac{9}{4}$ \\
\hline
$1$ & $1$ & $\begin{array}{l}\boxed{1}\boxed{3} \\\boxed{2}\end{array}$ & $0$ & $0$ & $-3$ & $3$ & $3$ & $-3$ & $0$ & $0$ & $3$ \\
\hline
$0$ & $3$ & $\begin{array}{l}\boxed{1}\boxed{2}\boxed{3}\end{array}$ & $1$ & $-3$ & $-3$ & $3$ & $-3$ & $3$ & $3$ & $-1$ & $5$ \\
\hline
\end{tabular}
\end{center}
\caption{The eight vectors $f_Q^{m, \ell}$ for $n = 3$. Observe that the vector $f_Q^{m, \ell}$ with $m + \ell = 0$ is an eigenvector of $K_3$ with eigenvalue $\beta_0 = 1$, the three vectors with $m + \ell = 2$ are eigenvectors with eigenvalue $\beta_1 = \frac{1}{4}$, and all other vectors are eigenvectors with eigenvalue $0$.}\label{examplefQvecs}
\end{figure}

Before beginning the proof, we make two important remarks for interpreting the content of the theorem.

\begin{remark}
There are two different sets of vectors here, $\{g_Q^{m, \ell}\}$ and $\{f_Q^{m, \ell}\}$ (with the same range of integers $m, \ell$ and corresponding tableaux $Q$). Both sets of vectors form orthogonal bases of $V^{\otimes n}$ with respect to $\pi$. We emphasize that the definition of the $g_Q^{m, \ell}$s does not depend on the transition matrix $K$ at all, and the proof of their orthogonality with respect to $\pi$ only uses the fact that $\pi$ is constant on each $\mc{O}_i$ (which is implied by the $S_n$-invariance). On the other hand, the $f_Q^{m, \ell}$s are defined as certain complicated linear combinations of the $g_Q^{m, \ell}$s. The orthogonality of these vectors also does not depend on the transition matrix, but these combinations are only eigenvectors because of \cref{fbyPhi} below, which makes explicit use of already-known eigenvectors of $K$.
\end{remark}

For our next remark, it will be useful to write out what those (non-orthogonal) eigenvectors of $K$ actually are. The following result is a restatement of \cite[Theorem 1.2 (3)]{diaconislinram1}, along with the remark at the end of Section 2.2 in that paper, in our current notation:

\begin{proposition}[\cite{diaconislinram1}]\label{eigenvectors}
For each subset $S \subseteq \{1, \ldots, n\}$, define
\[
    f_S = \sum_{T\subseteq \{1, \ldots, n\}} (-1)^{\vert S\cap T\vert} \binom{\vert S\vert}{\vert S\cap T\vert} v_T.
\]
Again recall the eigenvalues and multiplicities from \cref{unlumpedeigenvalues} above. Then $f_S$ is an eigenvector of $K_n$ of eigenvalue $\beta_k$ if $|S| = 2k$ is even, and it is an eigenvector of $K_n$ of eigenvalue $0$ otherwise. Furthermore, letting $\mathbb{S}(n)$ denote the set of all subsets of $\{1, \ldots, n\}$, $\{f_S: S \subseteq \mathbb{S}(n)\}$ is a basis of $V^{\otimes n}$.
\end{proposition}

\begin{remark}
The relation between the $f_Q^{m, \ell}$ and $g_Q^{m, \ell}$ vectors in Equation~\eqref{fQdef} can be made more transparent. Thanks to Lemma~\ref{fbyPhi} below, each $f$ is actually a linear combination of the $f_S$ vectors from Proposition~\ref{eigenvectors}, with coefficients given by the corresponding $g$. For instance, $g_{\scalebox{.8}{$\displaystyle\begin{array}{l}\boxed{1}\boxed{2}\boxed{3}\end{array}$}}^{0, 2}$ has a $1$ in each of the entries $011, 101$, and $110$, and indeed a direct calculation shows that $f_{\scalebox{.8}{$\displaystyle\begin{array}{l}\boxed{1}\boxed{2}\boxed{3}\end{array}$}}^{0, 2} = 1f_{\{2, 3\}} + 1f_{\{1, 3\}} + 1f_{\{1, 2\}}$.

Therefore, if there were another Markov chain in which we had a full basis of (not necessarily orthogonal) eigenvectors $f_S$ indexed by subsets $S$ of $\{1, \ldots, n\}$, which also satisfy $f_{wS} =wf_S$, then this same $g$-weighted combination of those vectors may also lead to orthogonal eigenvectors in that chain. However, one step of our proof relies on an explicit calculation (\cref{tableauorthogonality}) which is dependent on our choice of $f_S$, essentially because all $S$ of odd size correspond to the same eigenvalue of $0$. The full proof of \cref{orthoevthm} in the next section will also make these remarks more clear.
\end{remark}

\section{Proof of \cref{orthoevthm}}\label{schurweylproofsection}

We first prove a crucial lemma relating some of our $f$ and $g$ vectors to the previous $f_S$ eigenvectors from \cref{eigenvectors}.

\begin{lemma}\label{fbyPhi} 
Let $T$ be the column reading tableau of shape $(n-m,m)$ and let $f^{m,\ell}_{T}$ and $g_T^{m,\ell}$ be as defined in \cref{fgTd}; also recall the definition of $f_S$ from \cref{eigenvectors}. Define the vector space isomorphism $\Phi: V^{\otimes n} \to V^{\otimes n}$ by setting $\Phi(v_S) = f_S$ and extending by linearity. Then
\[
    f^{m,\ell}_{T} = \Phi(g_T^{m, \ell}).
\]
\end{lemma}
\begin{proof}
We begin by constructing careful linear combinations of our original eigenvectors $f_S$. For any subset $S\subseteq \{3,4, \ldots, n\}$, we have (here the notation $f_{01S}$ is shorthand for $f_{\{2\} \cup S}$)
\begin{align*}
f_{01S} &= \sum_{T\subseteq \{3, \ldots, n\}} 
(-1)^{\vert S\cap T\vert} \binom{1+\vert S\vert}{\vert S\cap T\vert}v_{00T}
+(-1)^{\vert S\cap T\vert} \binom{1+\vert S\vert}{\vert S\cap T\vert}v_{10T}
\\
&\qquad
+(-1)^{1+\vert S\cap T\vert} \binom{1+\vert S\vert}{1+\vert S\cap T\vert}v_{01T}
+(-1)^{1+\vert S\cap T\vert} \binom{1+\vert S\vert}{1+\vert S\cap T\vert}v_{11T}, 
\\
f_{10S} &= \sum_{T\subseteq \{3, \ldots, n\}} 
(-1)^{\vert S\cap T\vert} \binom{1+\vert S\vert}{\vert S\cap T\vert}v_{00T}
+(-1)^{1+\vert S\cap T\vert} \binom{1+\vert S\vert}{1+\vert S\cap T\vert}v_{10T} 
\\
&\qquad
+(-1)^{\vert S\cap T\vert} \binom{1+\vert S\vert}{\vert S\cap T\vert}v_{01T}
+(-1)^{1+\vert S\cap T\vert} \binom{1+\vert S\vert}{1+\vert S\cap T\vert}v_{11T}, 
\end{align*}
so subtracting these equations yields
\begin{align*}
f_{(01 - 10)S} = f_{01S}- f_{10S} 
&= \sum_{T\subseteq \{3, \ldots, n\}} 
(-1)^{\vert S\cap T\vert} \left( \binom{1+\vert S\vert}{\vert S\cap T\vert}+\binom{1+\vert S\vert}{1+\vert S\cap T\vert}\right)
v_{10T}
\\
&\qquad\qquad\qquad
-(-1)^{\vert S\cap T\vert} \left(\binom{1+\vert S\vert}{1+\vert S\cap T\vert} + \binom{1+\vert S\vert}{\vert S\cap T\vert}\right) v_{01T}
\\
&= (v_{01}-v_{10}) \otimes \left(\sum_{T\subseteq \{3, \ldots, n\}} 
(-1)^{1+\vert S\cap T\vert} \binom{2+\vert S\vert}{1+\vert S\cap T\vert}
v_T\right).
\end{align*}
Iterating this process on subsequent pairs of coordinates, we see that for any subset $S\subseteq \{2m+1, \ldots, n\}$, we have
\begin{align}\label{fTexp}
f_{(01-10)^m S}
&= (v_{01}-v_{10})^{\otimes m} \otimes \left( \sum_{T\subseteq \{2m+1, \ldots, n\} } (-1)^{m+\vert S\cap T\vert}\binom{2m+\vert S\vert}{m+\vert S\cap T\vert} v_T \right)
\nonumber \\
&= (v_{01}-v_{10})^{\otimes m} \otimes \left(\sum_{T\subseteq \{1, \ldots, n-2m\} } (-1)^{m+\vert S\cap T\vert}\binom{2m+\vert S\vert}{m+\vert S\cap T\vert} v_T\right),
\end{align}
where this last line is only a change in notation (notating a tensor product by starting both sets of coordinates from $1$). As before, let $\mathbb{S}(n-2m)$ denote the set of all subsets of $\{1, \ldots, n-2m\}$ and $\mathbb{S}(n-2m)_i$ the set of such subsets of cardinality $i$. Letting $T$ denote the column reading tableau of shape $(n-m,m)$ and defining 
\[
    g^{(i)}_{1\cdots (n-2m)} = \sum_{S\in \mathbb{S}(n-2m)_i} v_S
\]
(note that this is \textit{not} a vector of the form $g_T^{m,\ell}$), we have that 
\begin{align*}
f^{m, \ell}_{T} &= \sum_{i=0}^{n-2m} T_{m,n}^{(\ell)}(i) g^{m, i}_T \\
&= \sum_{i=0}^{n-2m} T_{m,n}^{(\ell)}(i) \left((v_{01}-v_{10})^{\otimes m} \otimes g^{(i)}_{1\cdots (n-2m)}\right) \\
&= (v_{01}-v_{10})^{\otimes m}
\otimes \left(\sum_{i=0}^{n-2m} g^{(i)}_{1\cdots(n-2m)}
T_{m,n}^{(\ell)}(i)\right) \\ 
&= (v_{01}-v_{10})^{\otimes m} \otimes \left(
\sum_{i=0}^{n-2m} g^{(i)}_{1\cdots (n-2m)}
\sum_{S \in \mathbb{S}(n-2m)_\ell}
(-1)^{m+\vert S\cap \{1,\ldots, i\} \vert }\binom{2m+\ell}{m+\vert S\cap \{1, \ldots, i\}  \vert} 
\right),
\end{align*}
where in the last line we plugged in the definition of our scalars $T_{m,n}^{(\ell)}(i)$. Therefore expanding out and then swapping the order of summation yields 
\begin{align*}
f^{m, \ell}_{T} &= (v_{01}-v_{10})^{\otimes m}\otimes \left( \sum_{i=0}^{n-2m} \sum_{T\in \mathbb{S}(n-2m)_i} v_T
\sum_{ S\in \mathbb{S}(n-2m)_\ell}
(-1)^{m+\vert S\cap \{1,\ldots, i\} \vert }\binom{2m+\ell}{m+\vert S\cap \{1, \ldots, i\}  \vert} \right) 
\\
&= \sum_{\substack{S\subseteq \{2m+1, \cdots, n\} \\ |S| = \ell}} \left( (v_{01}-v_{10})^{\otimes m} \otimes \sum_{T\in \mathbb{S}(n-2m)} (-1)^{m +\vert S\cap T\vert}\binom{2m+\vert S\vert}{m+\vert S\cap T\vert} v_T\right) \\
&= \sum_{S\in \mathbb{S}(n-2m)_\ell }  f_{(01-10)^mS},
\end{align*}
where before swapping summation in the middle step we use that the innermost sum over $S$ is always preserved if we replace $\{1, \cdots, i\}$ with the $i$-element subset $T$, and where the last equality follows from \cref{fTexp}. Now because $f_{(01 - 10)^m S}$ is just shorthand for a particular linear combination of the $f_S$s, we have $\Phi(v_{(01 - 10)^m S}) = f_{(01 - 10)^m S}$. Thus
\[
    f^{m, \ell}_{T} = \sum_{S\in \mathbb{S}(n-2m)_\ell }  f_{(01-10)^m S} = \sum_{S\in \mathbb{S}(n-2m)_\ell } \Phi(v_{(01-10)^m S}) =\Phi(g_T^{m, \ell}),
\]
completing the proof.
\end{proof}

We are now ready to prove \cref{orthoevthm}. Rephrasing the previous lemma, we have just shown that $f_T^{m, \ell}$ is a linear combination of eigenvectors $f_S$ each with $|S| = m + \ell$, and so our proof will show that many of the nice properties relating $f_T$ and $g_T$ are still preserved when we apply transpositions to the boxes of the Young tableau.

\begin{proof}[Proof of \cref{orthoevthm}]
Fix $m \in \{0,1,\ldots, \lfloor n/2\rfloor\}$, and let
$T$ be the column reading tableau of shape $(n-m,m)$. Also fix $i \in \{0,1,\ldots, n-2m\}$. We first claim that if $2 \le r \le n$, then (recall that $s_{tr}$ denotes the transposition switching $t$ and $r$)
\begin{align*}
    M_r g_T^{m, i} &= \sum_{t=1}^{r-1} s_{tr} \left(v_{01} - v_{10}\right)^{\otimes m} \otimes \left(\sum_{S \in \mathbb{S}(n-2m)_i} v_S\right) \\
    &= \mathrm{ct}(T(r))g_T^{m, i}.
\end{align*}
Indeed, this follows from an explicit calculation:
\begin{itemize}
    \item For $r > 2m$, we have $s_{tr} g_T^{m, i} = g_T^{m, i}$ for each $t \in \{2m+1, 2m+2, \ldots, r-1\}$, and we also have
    \[
        (s_{2a-1, r} + s_{2a, r}) g_T^{m, i} = g_T^{m, i} \quad \text{for all }a \in \{1, 2, \ldots, m\}
    \]
    since $s_{13} + s_{23}$ fixes both $(v_{010} - v_{100})$ and $(v_{011} - v_{101})$ (so the same holds with indices $2a-1, 2a, r$). Thus in total, $M_r g_T^{m, i} = (r - m - 1) g_T^{m, i}$.
    \item Similarly, for $r \le 2m$, we have $(s_{2a-1, r} + s_{2a, r}) g_T^{m, i} = g_T^{m, i}$ for all $a \in \{1, 2, \ldots, \lfloor \frac{r - 1}{2} \rfloor\}$ by the same calculation as above, and if $r$ is even, we also have $s_{r-1, r} g_T^{m, i} = -g_T^{m, i}$ because the corresponding $(v_{01} - v_{10})$ factor is negated. Thus in total, $M_r g_T^{m, i} = \left(\frac{r}{2} - 2\right) g_T^{m, i}$ for $r$ even and $M_r g_T^{m, i} = \frac{r - 1}{2} g_T^{m, i}$ for $r$ odd.
\end{itemize}
In all cases, the total coefficient of $g_T^{m, i}$ is indeed the column number of the box $r$ minus the row number. Therefore by definition of $V_T^{(m+i)}$, $g_T^{m, i} \in V_T^{(m+i)}$.

Now suppose $Q$ is a standard tableau of shape $(n-m,m)$ with $\boxed{a_1}\boxed{a_2}\boxed{\phantom{a_m}\cdots\phantom{a_m}}\boxed{a_m}$ in the second row. Let $\tau_Q \colon V_T^{(m+i)}\to V_Q^{(m+i)}$ be the vector space isomorphism defined in \cref{tauQd}. Since $g_Q^{m,i} = \tau_Q g_T^{m, i}$, 
\[
    g_Q^{m, i} \in V_Q^{(m+i)} \quad \text{and} \quad g_Q^{m, i}\ne 0.
\]
Thus, by \cref{lQperp} and \cref{Vndecomp},
\[
    \left\{ g_Q^{m, i}\ \middle|\ m\in \{0,\ldots, \lfloor n/2\rfloor\},\ Q\in \hat{S}^{(n-m,m)}_n,\ i\in \{0,\ldots, n-2m\} \right\}
\]
is an orthogonal basis of $V^{\otimes n}$ with respect to the inner product $\langle\cdot,\cdot\rangle_\pi$.

\medskip

We will now use the orthogonality of the $g_Q^{m, i}$s to also prove orthogonality of the $f_Q^{m,\ell}$s. Recall that $\Phi$ is the vector space isomorphism mapping $v_S$ to $f_S$ for all $S$. For any $w \in S_n$, we have
\[
    w\Phi(v_S) = wf_S = f_{wS} = \Phi(v_{wS}) = \Phi(wv_S).
\]
So $\Phi$ is an $S_n$-module isomorphism. Therefore \cref{fbyPhi} gives that
\[
    f^{m, \ell}_{Q} = \tau_Q f^{m, \ell}_T = \tau_Q \Phi(g_T^{m, \ell}) = \Phi(\tau_Q g_T^{m, \ell}) =\Phi(g_Q^{m, \ell})
\]
for all tableaux $Q$.
Since $\left\{ g_Q^{m,\ell}\ \middle|\ m\in \{0, \ldots, \lfloor n/2\rfloor\},\ Q\in \hat{S}_n^{(n-m,m)},\ \ell\in \{0,\ldots, n-2m\} \right\}$ is a basis of $V^{\otimes n}$, 
\[
    \left\{ f_Q^{m,\ell} \ \middle|\ m\in \{0, \ldots, \lfloor n/2\rfloor\},\ Q\in \hat{S}_n^{(n-m,m)},\ \ell\in \{0,\ldots, n-2m\} \right\} \quad \text{is a basis of }V^{\otimes n}.
\]
By \cref{eigenvectors}, $\{f_S\ |\ S\in \mathbb{S}(n)\}$ is a basis of $V^{\otimes n}$ where
\begin{equation}\label{feval}
f_S \text{ is an eigenvector of }K_n \text{ of eigenvalue} \begin{cases} \beta_{|S|/2}, & \text{ if }|S| \text{ is even}, \\ 0, & \text{otherwise.} \end{cases}
\end{equation}
So since $g_T^{m, \ell}$ is a linear combination of $v_S$s with $\vert S \vert = m+\ell$ and $wv_S = v_{wS}$ for all $w \in S_n$, $g_Q^{m, \ell}$ must also be a linear combination of $v_S$ with $\vert S \vert = m+\ell$. Thus $f^{m,\ell}_Q = \Phi(g_Q^{m,\ell})$ is a linear combination of $f_S$s all with $\vert S \vert = m+\ell$, meaning that \cref{feval} implies
\[
    f^{m, \ell}_Q \text{ is an eigenvector of }K_n \text{ with eigenvalue }\begin{cases} \beta_{(m+\ell)/2}, & \text{if }m + \ell \text{ is even}, \\ 0, & \text{otherwise.} \end{cases}
\]
We now prove that the eigenvectors $f_Q^{m, \ell}$ are all orthogonal. Because $\Phi$ is an $S_n$-module isomorphism, we have for any $r$ that
\[
    M_r f_Q^{m,\ell} = M_r \Phi(g_Q^{m,\ell}) = \Phi(M_r g_Q^{m,\ell}) = \Phi(\mathrm{ct}(Q(r)) g_Q^{m,\ell}) = \mathrm{ct}(Q(r))\Phi( g_Q^{m,\ell}) = \mathrm{ct}(Q(r))f_Q^{m,\ell}.
\]
So in fact $f_Q^{m,\ell}\in V_Q$, so it follows from \cref{Qperp} that
\[
    P \ne Q \implies \langle f^{m_1,\ell_1}_P,  f^{m_2,\ell_2}_Q\rangle_\pi = 0.
\] 
Therefore, it just remains to show orthogonality among eigenvectors labeled by the same tableau $Q$. Remembering that $m$ is determined by the shape of $Q$, it thus remains to show that $\langle f^{m,\ell_1}_Q, f^{m,\ell_2}_Q \rangle_\pi = 0$ for all $\ell_1 \ne \ell_2$. Because $K_n$ is self-adjoint with respect to $\langle \cdot, \cdot \rangle_\pi$, eigenvectors of $K_n$ of different eigenvalues are automatically orthogonal, but this does not account for the case where both vectors are eigenvectors with eigenvalue $0$.

For this, we claim first that for the column reading tableau $T$ of shape $(n-m, m)$, we have
\begin{align*}
    \langle f_T^{m, \ell_1}, f_T^{m, \ell_2} \rangle_\pi = 0
\end{align*}
for any $\ell_1 \ne \ell_2$. We defer the proof of this fact (which reduces to a ``WZ-pair'' binomial coefficient calculation done with computer assistance) to \cref{tableauorthogonality}. Then recalling that all transpositions are self-adjoint, each $\tau_j = s_j + \frac{1}{M_j - M_{j+1}}$ is also self adjoint. Therefore, for any tableau $Q$ and any transposition $s_i$ obtained in a sequence of the form in \cref{cQd}, we have
\begin{align*}
\langle f^{m, \ell_1}_{s_iQ}, f^{m, \ell_2}_{s_iQ}\rangle_\pi 
&= \langle \tau_i f^{m, \ell_1}_Q, \tau_i f^{m, \ell_2}_Q\rangle_\pi  \\
&= \langle f^{m, \ell_1}_Q, \tau_i^2 f^{m, \ell_2}_Q \rangle_\pi 
\\
&= 
\frac{ (\mathrm{ct}(Q(i))-\mathrm{ct}(Q(i+1))+1)(\mathrm{ct}(Q(i))-\mathrm{ct}(Q(i+1))-1)}{(\mathrm{ct}(Q(i))-\mathrm{ct}(Q(i+1)))^2}
\langle f^{m, \ell_1}_Q, f^{m, \ell_2}_Q\rangle_\pi,
\end{align*}
where the last equality follows from \cref{intprops} and the fact that $f_Q^{m, \ell} \in V_Q$. Since we obtain any tableau $Q$ from $T$ via a sequence of such transpositions, we find inductively that $\langle f^{m, \ell_1}_Q, f^{m, \ell_2}_Q\rangle_\pi = 0$ for all $Q$ of shape $(n-m, m)$. This completes the verification of orthogonality.

\medskip

Finally, we compute the norms of our eigenvectors $f_Q^{m,\ell}$. Again, let $T$ be the column reading tableau of shape $(n-m,m)$ and let $i \in \{0,1,\ldots, n-2m\}$. We can compute explicitly the norm of $g_T$, since
\begin{align}\label{gTnorm}
\langle g_T^{m, i}, g_T^{m, i}\rangle_\pi
&=\left\langle (v_{01}-v_{10})^{\otimes m} \otimes \Big(\sum_{S\in \mathbb{S}(n-2m)_i} v_S\Big), 
(v_{01}-v_{10})^{\otimes m} \otimes \Big(\sum_{S\in \mathbb{S}(n-2m)_i} v_S\Big) \right\rangle_\pi
\nonumber \\
&=\frac{1}{n+1} \sum_{S\in \mathbb{S}(n-2m)_i} \frac{1}{\binom{n}{m+i}}2^m \nonumber \\
&= \frac{2^m}{n+1} \frac{\binom{n-2m}{i}}{\binom{n}{m+i}},
\end{align}
where the middle step uses that $g_T^{m, i} \in V^{(m+i)}$ so that the inner product weight $\pi(x)$ is in fact constant. But then because the $g_T^{m, i}$s are orthogonal, our definition of $f_T^{m, \ell}$ as a linear combination yields
\[
    f^{m, \ell}_T = \sum_{i=0}^{n-2m} T_{m,n}^{(\ell)}(i) g^{m, i}_T \implies \langle f_T^{m, \ell}, f_T^{m, \ell}\rangle_\pi = \frac{2^m}{n+1} \sum_{i=0}^{n-2m} \big(T_{m,n}^{(\ell)}(i)\big)^2  \frac{\binom{n-2m}{i}}{\binom{n}{m+i}}.
\]
And now we again apply the logic for transferring inner products from $T$ to $Q$: for any tableau $Q$ and any transposition $s_i$ in the sequence of applied transpositions, we again have
\begin{align*}
\langle f^{m, \ell}_{s_iQ}, f^{m, \ell}_{s_iQ}\rangle_\pi 
&= \langle \tau_i f^{m, \ell}_Q, \tau_i f^{m, \ell}_Q\rangle_\pi  \\
&= \langle f^{m, \ell}_Q, \tau_i^2 f^{m, \ell}_Q \rangle_\pi 
\\
&= 
\frac{ (\mathrm{ct}(Q(i))-\mathrm{ct}(Q(i+1))+1)(\mathrm{ct}(Q(i))-\mathrm{ct}(Q(i+1))-1)}{(\mathrm{ct}(Q(i))-\mathrm{ct}(Q(i+1)))^2}
\langle f^{m, \ell}_Q, f^{m, \ell}_Q\rangle_\pi,
\end{align*}
and this time we extract the constants from \cref{tauQd}. For each $\tau_Q^{(r)} = \tau_{a_r - 1} \cdots \tau_{2r}$ performed from right-to-left, note that $\mathrm{ct}(Q(i))-\mathrm{ct}(Q(i+1))$ will successively take on the values $-2, -3, -4, \cdots, -(a_r - 2r + 1)$. Indeed, before we perform these transpositions, box $2r+1$ will always be one row above and one column to the right of box $2r$, and then any boxes $2r+2, 2r+3, \cdots$ involved in the transpositions will be immediately to the right of box $2r+1$. Therefore the product of all factors is exactly $\gamma_Q^{(r)}$, and multiplying all contributions together indeed yields
\[
    \langle f^{m, \ell}_{Q}, f^{m, \ell}_{Q}\rangle_\pi = \gamma_Q \langle f^{m, \ell}_{T}, f^{m, \ell}_{T}\rangle_\pi,
\]
completing the proof.
\end{proof}

\section{More on orthogonal polynomials and eigenvalue multiplicities}\label{schurweylsubsection2}

This section outlines some useful consequences and observations of \cref{orthoevthm}. We first describe a connection between the scalars $T_{m, n}^{(\ell)}(i)$ defined in \cref{Tscalarsdefn} and the discrete Chebyshev and Hahn polynomials. Using that the eigenvectors $f_T^{m,\ell}$ are orthogonal, we obtain formulas for values of Chebyshev and Hahn polynomials as sums involving binomial coefficients (\cref{actuallyhahn}), which then also enables an explicit closed-form formula for the norms of the eigenvectors $f_Q^{m, \ell}$ (\cref{hahnpolynomialnorms}). Finally, we obtain \cref{eigenvaluesplittingresult}, which refines \cref{eigenvectors} by describing the eigenvalue multiplicities when restricting $K_n$ to copies of each irreducible representation. 

For more background on orthogonal polynomial theory, see \cite{chihara} for a very readable exposition or \cite{koekoek} for a detailed compendium (we will follow the notational conventions of the latter). Define the rising factorial
\begin{equation}\label{risingfactorial}
    (a)_0 = 1, \quad (a)_j = a(a+1) \cdots (a+j-1).
\end{equation}
The \textbf{$(\alpha, \beta)$-Hahn polynomials} (see \cite[Section 1.5]{koekoek} for a definition in terms of the hypergeometric ${}_3 F_2$ function)
\[
    Q_{n;\alpha,\beta}^\ell(x) = \sum_{k=0}^n \frac{1}{k!} \frac{(-\ell)_k (\ell+\alpha+\beta+1)_k (-x)_k}{(\alpha+1)_k (-n)_k}
\]
are the orthogonal polynomials on $\{0, 1, \cdots, n\}$ with respect to the beta-binomial distribution $m(i) = \binom{n}{i} \frac{(\alpha+1)_i (\beta + 1)_{n-i}}{(\alpha + \beta + 2)_n}$, normalized so that $Q_{n;\alpha,\beta}^\ell(0) = 1$. (Note that this is a different convention from the one used in \cite[Section 2.3]{diaconiszhong}.) In particular, the $(0, 0)$-Hahn polynomials are exactly the discrete Chebyshev polynomials, since the corresponding beta-binomial distribution is uniform on $\{0, \cdots, n\}$.

\begin{proposition}\label{actuallyhahn}
The scalars $T_{m,n}^{(\ell)}(i)$ are related to the $(m, m)$-Hahn polynomials on $\{0, \cdots, n-2m\}$ as follows:
\[
    Q_{n-2m; m, m}^{\ell}(i) = \frac{(-1)^m}{\binom{n-2m}{\ell}\binom{2m+\ell}{m}} T_{m, n}^{(\ell)}(i).
\]
In particular, recall that $T^\ell_n(x)$ denotes the discrete Chebyshev polynomial on $\{0, 1, \cdots, n\}$ of degree $\ell$. Then the $m = 0$ case says that for all $0 \le j \le n$, 
\[
    T^\ell_n(j) = \frac{1}{\binom{n}{\ell}} T_{0, n}^{(\ell)}(j) = \frac{1}{\binom{n}{\ell}} \sum_{S \in \mathbb{S}(n)_\ell} (-1)^{|S \cap \{1, \cdots, j\}|} \binom{\ell}{|S \cap \{1, \cdots, j\}|}.
\]
\end{proposition}

\begin{proof}
First, write out \cref{Tscalarsdefn}:
\begin{align}\label{formulafortscalars}
    T^{(\ell)}_{m,n}(i) &= \sum_{S\in \mathbb{S}(n-2m)_\ell} (-1)^{m+\vert S\cap \{1, \ldots, i\}\vert} \binom{2m + \ell}{m+\vert S\cap \{1, \ldots, i\}\vert}\nonumber \\
    &= \sum_{j = 0}^i (-1)^{m+j} \binom{2m + \ell}{m+j} \cdot \binom{i}{j}\binom{n-2m-i}{\ell-j}.
\end{align}
Now fix $n$ and $m$. For each $\ell \in \{0, \cdots, n-2m\}$, \cref{formulafortscalars} is a polynomial in $i$ of degree $\ell$, since each term in the sum is of degree $\ell$ and has leading coefficient the same sign as $(-1)^{m+j} (-1)^{\ell - j} = (-1)^{m+\ell}$, which is constant for all terms. Moreover, because of the orthogonality guaranteed by \cref{orthoevthm}, if $\ell_1 \ne \ell_2$ then
\begin{align*}
    0 &= \langle f_T^{m, \ell_1}, f_T^{m, \ell_2} \rangle \\
    &= \sum_{i, j=0}^{n-2m} T_{m,n}^{(\ell_1)}(i)T_{m,n}^{(\ell_2)}(j) \langle g_T^{m, i}, g_T^{m, j} \rangle \\
    &= \sum_{i = 0}^{n-2m} T_{m,n}^{(\ell_1)}(i)T_{m,n}^{(\ell_2)}(i) \frac{2^m}{n+1} \frac{\binom{n-2m}{i}}{\binom{n}{m+i}}
\end{align*}
where we have used \cref{fgTd} and \cref{gTnorm} (along with the fact that the $g_T^{m, i}$s are orthogonal over $i$). Using the rising factorial notation of \cref{risingfactorial},
\begin{align*}
    \frac{\binom{n-2m}{i}}{\binom{n}{m+i}} &= \frac{(n-2m)!(m+i)!(n-m-i)!}{i!(n-2m-i)!n!} \\
    &= \frac{m!^2}{n!} \binom{n-2m}{i} \frac{(m+i)!}{m!} \frac{(n-m-i)!}{m!} \\
    &= \frac{m!^2}{n!} \binom{n-2m}{i} (m+1)_i (m+1)_{n-2m-i} 
\end{align*}
is proportional in $i$ to the beta-binomial distribution on $\{0, \cdots, n-2m\}$ with parameters $(\alpha, \beta) = (m, m)$. So our calculation above shows that $T_{m,n}^{(\ell_1)}$ and $T_{m,n}^{(\ell_2)}(i)$ are orthogonal with respect to the beta-binomial distribution. Thus up to a constant, $T_{m, n}^{(\ell)}$ is exactly the degree-$\ell$ $(m, m)$-Hahn polynomial, and we can evaluate the constant by evaluating at $i = 0$. We know that $Q_{n-2m;m,m}^{(\ell)}(0) = 1$, and plugging in $i = 0$ into \cref{formulafortscalars} leaves only the $j = 0$ term, which is $(-1)^m \binom{2m+\ell}{m} \binom{n-2m}{\ell}$. Dividing by this constant yields the result.
\end{proof}

This connection to orthogonal polynomials lets us write the sum appearing in \cref{fmlqnorm} in closed form, since orthogonality relations for such polynomials are readily available. Indeed, in our notation, the orthogonality equation \cite[(1.5.2)]{koekoek} reads, for any $\alpha, \beta > -1$,
\begin{equation}\label{hahnorthogonality}
    \sum_{i=0}^n \binom{\alpha+i}{i} \binom{n+\beta-i}{n-i} Q_{n;\alpha,\beta}^{\ell}(i) Q_{n;\alpha,\beta}^{\ell'}(i) = \delta_{\ell\ell'} \frac{(-1)^\ell \ell! (\beta + 1)_\ell (\ell + \alpha + \beta + 1)_{n+1}}{n! (2\ell + \alpha + \beta + 1)(-n)_\ell (\alpha + 1)_\ell}.
\end{equation}

\begin{corollary}\label{hahnpolynomialnorms}
Using the notation in \cref{orthoevthm}, with $\gamma_Q$ as defined in \cref{gmQ},
\[
    \langle f^{m, \ell}_Q, f^{m, \ell}_Q\rangle_{\pi} = \gamma_Q\frac{2^m}{n+1} \cdot \frac{(2m+\ell)!}{(2m + 2\ell + 1) (m+\ell)!^2\ell!} \cdot \frac{(n-2m)!}{n!} \cdot \frac{(n + \ell+1)!}{(n-2m-\ell)!}.
\]
\end{corollary}

It may be easier to parse this formula by treating $m$ and $\ell$ as constants and considering the dependence on $n$. In particular, for the column reading tableau $T$ (with $\gamma_T = 1$), setting $m = 0$ yields
\begin{align}\label{hahnnorm0}
    \langle f_T^{0, \ell}, f_T^{0, \ell} \rangle_{\pi} &= \frac{1}{n+1} \cdot \frac{\ell!}{(2\ell+1)\ell!^3} \frac{(n+\ell+1)!}{(n-\ell)!}\nonumber \\
    &= \frac{1}{(2\ell+1)\ell!^2} \cdot \frac{1}{n+1} \prod_{i=-\ell+1}^{\ell+1} (n+i)
\end{align}
and similarly setting $m = 1$ yields
\begin{align}\label{hahnnorm1}
    \langle f_T^{1, \ell}, f_T^{1, \ell} \rangle_{\pi} &= \frac{2}{n+1} \cdot \frac{(\ell+2)!}{(2\ell+3)(\ell+1)!^2\ell!} \cdot \frac{1}{n(n-1)} \cdot \frac{(n+\ell+1)!}{(n-2-\ell)!}\nonumber \\
    &= \frac{2(\ell+2)}{(\ell+1)(2\ell+3)\ell!^2} \cdot \frac{1}{(n-1)n(n+1)} \prod_{i=-\ell-1}^{\ell+1} (n+i).
\end{align}
The key interpretation is that these expressions are always rational functions of $n$ which nicely factor into terms of the form $(n+i)$. (We will use these later on to bound the complicated expression for chi-square distance in \cref{onlytwocontributions}.)

\begin{proof}
Plugging in $\alpha = \beta = m, \ell = \ell'$, and replacing $n$ with $n-2m$ in \cref{hahnorthogonality} yields
\begin{align*}
    \sum_{i=0}^{n-2m} \binom{m+i}{i} \binom{n-m-i}{n-2m-i} Q_{n-2m;m,m}^{\ell}(i)^2 &= \frac{(-1)^\ell \ell! (\ell + 2m + 1)_{n-2m+1}}{(n-2m)! (2\ell + 2m + 1)(-(n-2m))_\ell} \\
    &= \frac{\ell + 2m + 1}{2\ell + 2m + 1}\frac{\binom{\ell+n+1}{n-2m}}{\binom{n-2m}{\ell}}.
\end{align*}
Thus the sum in \cref{fmlqnorm} can be written as 
\begin{align*}
    \sum_{i=0}^{n-2m} \big(T_{m,n}^{(\ell)}(i)\big)^2 \frac{\binom{n-2m}{i}}{\binom{n}{m+i}} &= \binom{n-2m}{\ell}^2 \binom{2m+\ell}{m}^2 \cdot\sum_{i=0}^{n-2m} \big(Q_{n-2m;m,m}^{(\ell)}(i)\big)^2 \frac{\binom{n-2m}{i}}{\binom{n}{m+i}} \\
    &= \binom{n-2m}{\ell}^2 \binom{2m+\ell}{m}^2\cdot\sum_{i=0}^{n-2m} \big(Q_{n-2m;m,m}^{(\ell)}(i)\big)^2 \frac{(n-2m)!(m+i)!(n-m-i)!}{i!(n-2m-i)!n!} \\
    &= \binom{n-2m}{\ell}^2 \binom{2m+\ell}{m}^2\cdot\frac{(n-2m)!m!^2}{n!}\sum_{i=0}^{n-2m} \big(Q_{n-2m;m,m}^{(\ell)}(i)\big)^2 \binom{m+i}{i} \binom{n-m-i}{n-2m-i} \\
    &= \binom{n-2m}{\ell}^2 \binom{2m+\ell}{m}^2 \frac{(n-2m)!m!^2}{n!}  \frac{\ell + 2m + 1}{2m + 2\ell + 1}\frac{\binom{\ell+n+1}{n-2m}}{\binom{n-2m}{\ell}} \\
    &=  \frac{(n-2m)!m!^2}{n!}  \frac{\ell + 2m + 1}{2m + 2\ell + 1} \binom{2m+\ell}{m}^2 \binom{n-2m}{\ell}\binom{\ell+n+1}{n-2m} \\
    &= \frac{(2m+\ell)!}{(2m+ 2\ell + 1) (m+\ell)!^2\ell!} \cdot \frac{(n-2m)!}{n!} \cdot \frac{(n + \ell+1)!}{(n-2m-\ell)!},
\end{align*}
and multiplying by the additional factors $\gamma_Q \frac{2^m}{n+1}$ in \cref{fmlqnorm} yields the result.
\end{proof}

\medskip

We'll now elaborate on how this eigenvector decomposition sheds light on the structure of the eigenspaces for the binary Burnside process. We claimed in \cref{schurweylsplittingmechanism} that the subspaces that our $f^{m,\ell}_Q$s span are exactly invariant subspaces under certain actions, and we elaborate now. On each ``orbit level'' subspace $V^{(i)}$ for $i \le \frac{n}{2}$, the action of $S_n$ permutes the locations of the $v_1$s (equivalently, the coordinates of the ones in the $n$-tuple), so that $V^{(i)}$ is isomorphic to the permutation representation $M^{(n-i, i)} = \text{Ind}_{S_i \times S_{n-i}}^{S_n}(1)$ on size-$i$ subsets. (Similarly for $i > \frac{n}{2}$, the action permutes the locations of the zeros and thus is isomorphic to $M^{(i, n-i)}$). These permutation representations decompose into irreducible representations as 
\begin{equation}\label{irreddecomp}
    M^{(n-i, i)} = S^{(n)} \oplus S^{(n-1, 1)} \oplus \cdots \oplus S^{(n-i, i)},
\end{equation}
where for any partition $\lambda$ of $n$, $S^\lambda$ is the irreducible \textit{Specht module} associated to that partition. (\cref{irreddecomp} is a special case of Young's rule, or more generally the Littlewood--Richardson rule \cite[Sections 14-17]{jamessymmetric}.) In particular, this means $V^{\otimes n}$ contains $(n+1)$ copies of $S^{(n)}$, $(n-1)$ copies of $S^{(n-1, 1)}$, $(n-3)$ copies of $S^{(n-2, 2)}$, and so on. 

The fact that the binary Burnside transition matrix $K_n$ commutes with the action of $S_n$ implies (by Schur's lemma) that as a map from any $S^{\lambda}$ to any $S^{\mu}$, $K_n$ must act as a constant multiple of the identity, and that constant is only nonzero if $\lambda = \mu$. This alone already implies that eigenvalues will appear with high multiplicity, since restricting $K_n$ to the copies of $S^{(n-i, i)}$ yields an operator whose eigenvalues repeat with multiplicity $\dim(S^{(n-i, i)}) = \binom{n}{i} - \binom{n}{i-1}$.

There is an additional algebraic structure that allows us to partition the eigenvalues into subsets finer than the partition by copies of $S^{(n-i, i)}$. This is how \cref{schurweylsplittingmechanism} makes its appearance. The Lie algebra $\mathfrak{sl}_2$ of $2 \times 2$ matrices with trace $0$ is spanned by $\{e, f, h\}$, where
\begin{equation}\label{efhdefn}
    e = \begin{bmatrix} 0 & 1 \\ 0 & 0 \end{bmatrix}, \quad f = \begin{bmatrix} 0 & 0 \\ 1 & 0 \end{bmatrix}, \quad h = \begin{bmatrix} 1 & 0 \\ 0 & -1 \end{bmatrix}.
\end{equation}
So $\mathfrak{sl}_2$ acts on the two-dimensional vector space $V$ spanned by $\{v_0, v_1\}$ via matrix multiplication, and it acts on $V^{\otimes n}$ by
\[
    g(v_1 \otimes \cdots \otimes v_n) = \sum_{i=1}^n v_1 \otimes \cdots \otimes v_{i-1} \otimes gv_i \otimes v_{i+1} \otimes \cdots \otimes v_n, \quad \text{for }g \in \mathfrak{sl}_2.
\]
The statement of Schur--Weyl duality is that this action and the $S_n$ action commute, and each is the full centralizer of the other, leading to the decomposition 
\begin{equation}\label{schurweyldecomp}
    L(C_2^n) = \bigoplus_{\lambda} S^{\lambda} \otimes L^{\lambda}
\end{equation}
where the $L^{\lambda}$s are irreducible representations of $\mathfrak{sl}_2$ and the $S^{\lambda}$s are irreducible representations of $S_n$. The isotypic components $S^{\lambda} \otimes L^{\lambda}$ in the decomposition in \cref{schurweyldecomp} are the irreducible $\mathfrak{sl}_2$-invariant and $S_n$-invariant subspaces of $L(C_2^n)$. The Schur--Weyl duality and the fact that the action of $K_n$ on $L(C_2^n)$ commutes with the action of $S_n$ implies that there exists an element of the universal enveloping algebra of $\mathfrak{sl}_2$ (a polynomial in the $e$s, $f$s, and $h$s) that acts on $L(C_2^n)$ the same way that $K_n$ does. This is discussed further in \cref{pplusexpression} below.

\cref{orthoevthm} provides eigenvectors that lie in the isotypic components $S^{\lambda} \otimes L^{\lambda}$ and implies the following description specifying how the eigenvalue multiplicities manifest across the different irreducible subspaces.

\begin{corollary}\label{eigenvaluesplittingresult}
Let $K_n^{\lambda}$ be the binary Burnside operator $K_n$ restricted to the isotypic component $S^{\lambda} \otimes L^{\lambda}$. By Schur's lemma, $K_n^{\lambda}$ acts as $ I_{\dim(S^{\lambda})} \otimes \tilde{K}_n^\lambda$ for some operator $\tilde{K}_n^\lambda$ (which is typically not a Markov chain). Let $\beta_k = \frac{1}{2^{4k}} \binom{2k}{k}^2$ be as in \cref{eigenvectors}. Then for any $k \le \frac{n}{2}$, $\beta_k$ is an eigenvalue of $\tilde{K}_n^\lambda$ of multiplicity $1$ (meaning it is an eigenvalue of $K_n^\lambda$ of multiplicity $\dim(S^{\lambda})$) for $\lambda = (n-m, m)$ when $m \le \min(2k, n-2k)$, and there are no other nonzero eigenvalues.
\end{corollary}

The decompositions for $n = 4$ and $n = 5$ are shown below for illustration. Note that $\dim(S^{(n-m, m)})$ is exactly the number of standard Young tableaux of shape $(n-m, m)$, so that each value in the array corresponds to a particular eigenvector. The rows with an even value of $m + \ell$ correspond to eigenspaces for the nonzero eigenvalues. For $n = 4$ and $n = 5$, respectively, we have the following tables:
\[
    \begin{array}{c|cccccccccc}
m+\ell 
&\sTBL{{1,2},{3,4}}
&\sTBL{{1,3},{2,4}}
&&\sTBL{{1,2,3},{4}}
&\sTBL{{1,2,4},{3}}
&\sTBL{{1,3,4},{2}}
&&\sTBL{{1,2,3,4}}
\\
\hline
0& &&&&&&&1 \\
1& &&&0 &0 &0 &\quad &0 \\
2& \frac14 &\frac14 &\quad &\frac14 &\frac14 &\frac14 &&\frac14 \\
3& &&&0 &0 &0 & &0 \\
4& &&&&&&&\frac{9}{64}
\end{array}
\]
\[
    \begin{array}{c|cccccccccccccc}
m+\ell &\sTBL{{1,2,3},{4,5}}
&\sTBL{{1,2,4},{3,5}}
&\sTBL{{1,3,4},{2,5}}
&\sTBL{{1,3,5},{2,4}}
&\sTBL{{1,2,5},{3,4}}
&&\sTBL{{1,2,3,4},{5}}
&\sTBL{{1,2,3,5},{4}}
&\sTBL{{1,2,4,5},{3}}
&\sTBL{{1,3,4,5},{2}}
&&\sTBL{{1,2,3,4,5}}
\\
\hline
0 &&&&&&&&&&&&1 \\
1& &&&&&&0 &0 &0 &0 &\quad &0 \\
2&\frac14 &\frac14 &\frac14 &\frac14 &\frac14 &\quad 
 &\frac14 &\frac14 &\frac14 &\frac14 &&\frac14 \\
3 &0 &0 &0 &0 &0 &&0 &0 &0 &0 &&0 \\
4 &&&&& &&\frac{9}{64} &\frac{9}{64}&\frac{9}{64} &\frac{9}{64} &&\frac{9}{64} \\
5& &&&&&&&&&&&0
\end{array}
\]
\begin{proof}
\cref{orthoevthm} exhibits a full basis of eigenvectors $\{f_Q^{m, \ell}\}$ with the eigenvectors indexed by the $\dim(S^{(n-m, m)})$ Young tableaux of shape $(n-m, m)$ in the isotypic component $S^{(n-m, m)} \otimes L^{(n-m, m)}$. The eigenvalue of an eigenvector $f_Q^{m, \ell}$ depends only on $m + \ell$. Since $\ell \in \{0, 1, \cdots, n- 2m\}$, the eigenvectors with $m + \ell \in \{m, m+1, \cdots, n-m\}$ appear in the isotypic component $S^{(n-m, m)} \otimes L^{(n-m, m)}$. This means that $\beta_k$ appears as an eigenvalue in the isotypic component $S^{(n-m, m)} \otimes L^{(n-m, m)}$ if and only if $m \le 2k \le n-m$, or equivalently, if and only if $m \le \min(2k, n-2k)$.
\end{proof}

\begin{remark}
For any representation $\rho$ of a finite group $G$ on a vector space $V$ and any irreducible character $\chi$ of $G$ of degree $d$, the projection of $\rho$ onto the copies of $\chi$ that appear is given by
\[
    P = \frac{d}{|G|} \sum_{g \in G} \chi(g^{-1}) \rho(g).
\]
(This formula is the ``canonical decomposition'' described in \cite[Section 2.6]{serrerepresentations}.) In particular, if $\chi$ corresponds to the trivial representation, $P$ is exactly averaging over the entire orbit, so $K_n$ restricted to the $(n+1)$ copies of the trivial representation is exactly $K^{\text{lumped}}$; this is consistent with the fact that all eigenvalues $\beta_0, \beta_1, \cdots, \beta_{\lfloor n/2 \rfloor}$ each appear exactly once in the lumped chain. The fact that Hahn polynomials appear in the eigenvectors corresponding to the other irreducible representations indicates that there may be nice interpretations for the other projections as well, even if they are not Markov operators.
\end{remark}

\section{Mixing time analysis from the one-ones state}\label{schurweylsubsection3}

We will now describe a probabilistic application of the orthogonal basis of eigenfunctions $\{f_Q^{m, \ell}\}$, showing that a bounded number of steps also suffices from the ``all-but-one zeros'' state by evaluating the $f_Q^{m, \ell}$s at this state (and in fact finding that almost all of them evaluate to zero, in stark contrast with our original basis $\{f_S\}$). For this, recall the expression for $\chi_x^2(\ell)$ from \cref{l1vsl2bound}. In our new notation, we have that the chi-square distance to stationarity started from $x$ after $s$ steps is
\begin{equation}\label{messychisquare}
    \chi_x^2(s) = \sum_{\substack{m \in \{0, \cdots, \lfloor n/2 \rfloor\} \\ Q \in \hat{S}_n^{(n-m, m)} \\ \ell \in \{0, \cdots, n-2m\} \\ m + \ell \text{ even} \\(m,\ell) \ne (0, 0)}} \left(\overline{f_Q^{m, \ell}}(x)\right)^2 \left(\beta_{(m+\ell)/2}\right)^{2s},
\end{equation}
where $\overline{f_Q^{m, \ell}}(x)$ is the $\ell^2(\pi)$-normalized multiple of $f_Q^{m, \ell}$ (whose squared norm is given by \cref{fmlqnorm}). Of course, this can be a very complicated sum in general, but we will now demonstrate that the specific form of our basis $\{f_Q^{m, \ell}\}$ can be very convenient for computations. To do this, we will compute the chi-square distance to stationarity started from the ``one-ones'' state $e_n$, where
\[
    e_j  = (0, \cdots, 0, 1, 0, \cdots, 0) \text{ has a }1 \text{ in only the }j\text{th coordinate}.
\]
To apply \cref{messychisquare}, we will need to compute the value of each $f_Q^{m, \ell}$ at this state.

\begin{remark}
By symmetry of the original Markov chain under permutation of coordinates, we know that the chi-square distance to stationarity is identical for any of the ``one-ones'' states. However, the individual values of $f_Q^{m, \ell}(x)$ are different for those different states $x$, and our choice to use $e_n$ will make the computation easier.
\end{remark}

\begin{proof}[Proof of \cref{swdistancefrom1}]
First, let $S = \{n\}$, so that $v_S =  v_0^{\otimes (n-1)} \otimes v_1$. We note that by definition, for any $v \in V^{\otimes n}$ we have that
\begin{align*}
    v(e_n) &= \text{the coefficient of }v_S \text{ in }v \\
    &= (n+1)n \langle v, v_S \rangle_\pi,
\end{align*}
since our stationary distribution assigns mass $\frac{1}{(n+1) \binom{n}{1}}$ to the state $e_n$. Thus, we may alternatively think of evaluating our vectors at $e_n$ as computing the inner products $n(n+1)\langle f_Q^{m, \ell}, v_S \rangle_\pi$, though we will not take this perspective here.

Since $f_Q^{m,\ell}$ is a linear combination of the $g_Q^{m, i}$s, it will be simplest to compute using the $g_Q^{m, i}$s first. As usual, we begin with $T$, the column reading tableau of shape $(n-m, m)$. By inspection of the definition of $g_T^{m, i}$, 
\begin{equation}\label{evalgTatsingleton}
    g_T^{m, i}(e_j) = \begin{cases} 1, & \text{if }(m, i) = (0, 1), \\ -1, & \text{if } (m, i) = (1, 0) \text{ and }j = 1, \\ 1, & \text{if } (m, i) = (1, 0) \text{ and }j = 2, \\ 0, & \text{otherwise.}\end{cases}
\end{equation}
The key property used here is that for $m \ge 2$, the $(v_{01} - v_{10})^{\otimes m}$ term means that all nonzero terms have at least two coordinates with $v_1$s and so the coefficient of $v_S$ (as well as all other singleton sets) must be zero. Since $f_T^{m, \ell}$ is a linear combination of $g_T^{m, i}$s, if $m \ge 2$, then $f_T^{m, \ell}(e_j) = 0$. Recalling \cref{fQdef}, since each $g_Q^{m, i}$ (resp. $f_Q^{m, \ell}$) is a linear combination of permuted $g_T^{m, i}$s (resp. $f_T^{m, \ell}$s), this immediately implies that for all $1 \le j \le n$ (and in particular $j = n$),
\begin{equation}\label{mge2allzero}
    f_Q^{m, \ell}(e_j) = 0 \text{ for all } m \ge 2 \text{ (and all  }\ell \in \{0, \cdots, n-2m\}\text{ and }Q \in \hat{S}_n^{(n-m, m)}).
\end{equation}
Thus we need only compute $f_Q^{m, \ell}(e_n)$ for $m = 0, 1$. The case $m = 0$ is simpler, since the only tableau of shape $(n-0, 0)$ is the column reading tableau $T$. If $0 \le \ell \le n$ and $1 \le j \le n$, then
\begin{align}\label{numerator0}
    f_T^{0, \ell}(e_j) &= \sum_{i=0}^{n} T_{0,n}^{\ell}(i) g_T^{0, i}(e_1) \nonumber \\
    &= T_{0,n}^{\ell}(1) g_T^{0, 1}(e_1) \nonumber \\
    &= \sum_{S \in \mathbb{S}(n)_\ell} (-1)^{|S \cap \{1\}|} \binom{\ell}{|S \cap \{1\}|} \nonumber \\
    &= \binom{n-1}{\ell} - \ell\binom{n-1}{\ell-1}.
\end{align}
Next, we do the case $m = 1$. A calculation for the column reading tableau $T$ of shape $(n-1, 1)$ gives
\begin{align*}
    f_T^{1, \ell}(e_1) &= \sum_{i=0}^{n-2} T_{1, n}^\ell(i) g_T^{1, i}(e_1) \\
    &= T_{1, n}^\ell(0) g_T^{1, 0}(e_1) \\
    &= -\sum_{S \in \mathbb{S}(n-2)_\ell} (-1)^{1+0}\binom{2 + \ell}{1+0} \\
    &= (2+\ell) \binom{n-2}{\ell}.
\end{align*}
By similar calculations using \cref{evalgTatsingleton}, 
\[
    f_T^{1, \ell}(e_2) = -(2+\ell) \binom{n-2}{\ell}, \quad f_T^{1,\ell}(e_j) = 0 \text{ for }j \ge 3.
\]
A general tableau $Q$ of shape $(n-1, 1)$ has a single entry $a_1 \ge 2$ in the second row, and so the definition in \cref{tauQd} gives
\begin{align*}
    f_Q^{1, \ell} &= \tau_{a_1 - 1} \cdots \tau_2 f_T^{1, \ell}\\
    &= \left(s_{a_1 - 1} - \frac{1}{a_1 - 1}\right) \cdots \left(s_2 - \frac{1}{2}\right) f_T^{1, \ell}.
\end{align*}
We can now see why picking a particular state $e_j$ can simplify calculations; we will demonstrate this by evaluating our vectors at $e_1$ and at $e_n$. First of all, the transpositions $s_2, \cdots, s_{a_1 - 1}$ do not alter the value at $e_1$, so
\begin{align*}
    f_Q^{1, \ell}(e_1) &= \left(1 - \frac{1}{a_1 -1}\right) \cdots \left(1 - \frac{1}{2}\right) f_T^{1, \ell}(e_1) \\
    &= \frac{1}{a_1 - 1} f_T^{1, \ell}(e_1) \\
    &= \frac{1}{a_1 - 1} (2+\ell) \binom{n-2}{\ell}.
\end{align*}
Thus any tableau $Q$ of shape $(n-1, 1)$ contributes to the final sum in \cref{messychisquare}. On the other hand, take any $n \ge 3$. If we instead choose to evaluate $f_Q^{1, \ell}$ at $e_n$, then the only nonzero contribution comes from applying the sequence of transpositions $s_2, s_3, \cdots, s_{n-1}$ to get from $T$ to $Q$ (because $f_T^{1, \ell}(e_j)$ is only nonzero for $j = 1, 2$). In other words,
\begin{equation}\label{numerator1}
    f_Q^{1, \ell}(e_n) = \begin{cases} -(2+\ell) \binom{n-2}{\ell} & \text{if }a_1 = n, \\ 0 & \text{otherwise}. \end{cases}
\end{equation}
Thus evaluating the eigenvectors $f_Q^{m, \ell}$ at $e_n$ results in the smallest number of nonzero terms in \cref{messychisquare}. We find that if $n \ge 3$, then
\begin{align}\label{onlytwocontributions}
    \chi_{e_n}^2(s) &= \sum_{\substack{m \in \{0, \cdots, \lfloor n/2 \rfloor\} \\ Q \in \hat{S}_n^{(n-m, m)} \\ \ell \in \{0, \cdots, n-2m\} \\ m + \ell \text{ even}\\(m,\ell) \ne (0, 0)}} \left(\overline{f_Q^{m, \ell}}(e_n)\right)^2 \left(\beta_{(m+\ell)/2}\right)^{2s} \nonumber \\
    &= \sum_{\ell \in \{1, \cdots, n\} \text{ even}} \frac{f_{Q_{(0)}}^{0,\ell}(e_n)^2}{\langle f_{Q_{(0)}}^{0, \ell}, f_{Q_{(0)}}^{0, \ell} \rangle_\pi}(\beta_{\ell/2})^{2s} + \sum_{\ell \in \{0, \cdots, n-2\} \text{ odd}} \frac{f_{Q_{(1)}}^{1,\ell}(e_n)^2}{\langle f_{Q_{(1)}}^{1, \ell}, f_{Q_{(1)}}^{1
    , \ell} \rangle_\pi}(\beta_{(1+\ell)/2})^{2s},
\end{align}
where the two sums are the contributions from $m = 0, 1$ respectively, $Q_{(0)}$ is the only tableau of shape $(n)$, and $Q_{(1)}$ is the tableau of shape $(n-1, 1)$ with $n$ in the second row. We have just computed all of the values in the numerators of the fractions, and we can use \cref{hahnpolynomialnorms} to compute the denominators. By \cref{hahnnorm0}, we have 
\begin{equation}\label{denominator0}
    \langle f_{Q_{(0)}}^{0, \ell}, f_{Q_{(0)}}^{0, \ell} \rangle_\pi = \langle f_T^{0, \ell}, f_T^{0, \ell} \rangle_{\pi} = \frac{1}{(2\ell+1)\ell!^2} \cdot \frac{1}{n+1} \prod_{i=-\ell+1}^{\ell+1} (n+i).
\end{equation}
Similarly by \cref{hahnnorm1}, and using that $\gamma_{Q_{(1)}} = \left(1 - \frac{1}{(n-1)^2}\right) \cdots \left(1 - \frac{1}{2^2}\right) = \frac{n}{2(n-1)}$ by a telescoping product,
\begin{equation}\label{denominator1}
    \langle f_{Q_{(1)}}^{1, \ell}, f_{Q_{(1)}}^{1, \ell} \rangle_\pi = \frac{n}{2(n-1)} \langle f_T^{1, \ell}, f_T^{1, \ell} \rangle_{\pi} = \frac{(\ell+2)}{(\ell+1)(2\ell+3)\ell!^2} \cdot \frac{1}{(n-1)^2(n+1)} \prod_{i=-\ell-1}^{\ell+1} (n+i).
\end{equation}

\medskip 

We will now compute a lower bound by taking only the term $\ell = 2$ from the first sum and $\ell = 1$ from the second sum in \cref{onlytwocontributions}. (Indeed, $(m, \ell) = (0, 2)$ and $(1, 1)$ are the only terms corresponding to the largest nontrivial eigenvalue $\beta_1$.) The numerators here are the squares of $f_{Q_{(0)}}^{0,2}(e_n) = \frac{(n-1)(n-6)}{2}$ and $f_{Q_{(1)}}^{1,1}(e_n) = -3(n-2)$, and the denominators simplify to
\[
    \langle f_{Q_{(0)}}^{0, 2}, f_{Q_{(0)}}^{0, 2} \rangle_\pi = \frac{1}{20} \cdot \frac{1}{n+1} \prod_{i=-1}^3 (n+i) = \frac{1}{20}(n-1)n(n+2)(n+3)
\]
and
\[
    \langle f_{Q_{(1)}}^{1, 1}, f_{Q_{(1)}}^{1, 1} \rangle_\pi = \frac{3}{10} \cdot \frac{1}{(n-1)^2(n+1)} \prod_{i=-2}^2 (n+i) = \frac{3}{10}\frac{(n-2)n(n+2)}{n-1}.
\]
The most important feature to notice is that the numerator and denominator of the fractions in \cref{onlytwocontributions} are polynomials of the same order. Thus, the lower bound
\[
    \chi_{e_n}^2(s)\ge \left(\frac{\left(\frac{(n-1)(n-6)}{2}\right)^2}{\frac{1}{20}(n-1)n(n+2)(n+3)} + \frac{(-3(n-2))^2}{\frac{3}{10} \frac{(n-2)n(n+2)}{n-1}}\right) \left(\frac{1}{4}\right)^{2s}
\]
is asymptotically $35 \cdot \left(\frac{1}{4}\right)^{2s}$, and in particular it is at least $5 \cdot \left(\frac{1}{4}\right)^{2s}$ for all $n \ge 3$.

\medskip

Finally, for the upper bound, we include all terms and upper bound each one independently of $n$. Plugging in \cref{numerator0} and \cref{denominator0}, we have
\begin{align*}
    \frac{f_{Q_{(0)}}^{0,\ell}(e_n)^2}{\langle f_{Q_{(0)}}^{0, \ell}, f_{Q_{(0)}}^{0, \ell} \rangle_\pi} &= (n+1)(2\ell+1)\ell!^2\frac{\left(\binom{n-1}{\ell} - \ell \binom{n-1}{\ell-1}\right)^2}{\prod_{i=-\ell+1}^{\ell+1} (n+i)} \\
    &\le (2\ell+1)\frac{(n+1)\prod_{i=1}^{\ell} (n-i)^2}{{\prod_{i=-\ell+1}^{\ell+1} (n+i)}} + (2\ell+1)\ell^4\frac{(n+1)\prod_{i=1}^{\ell-1} (n-i)^2}{{\prod_{i=-\ell+1}^{\ell+1} (n+i)}}
\end{align*}
using that $(a-b)^2 \le a^2 + b^2$ for $a, b > 0$. But now we can pair up each linear factor in $n$ in the numerators with a larger factor in the denominator, meaning that this is simply upper bounded by $(2\ell+1) + \frac{(2\ell+1) \ell^4}{n^2}\le (2\ell+1) + (2\ell+1)\ell^2 = 2\ell^3+\ell^2+2\ell+1$. Similarly, plugging in \cref{numerator1} and \cref{denominator1}, 
\begin{align*}
    \frac{f_{Q_{(1)}}^{1,\ell}(e_n)^2}{\langle f_{Q_{(1)}}^{1, \ell}, f_{Q_{(1)}}^{1, \ell} \rangle_\pi} &= \frac{(\ell+1)(2\ell+3)\ell!^2}{(\ell+2)} \cdot \frac{(n-1)^2(n+1)(2+\ell)^2 \binom{n-2}{\ell}^2}{\prod_{i=-\ell-1}^{\ell+1} (n+i)} \\
    &= (\ell+1)(\ell+2)(2\ell+3) \frac{(n-1)^2(n+1) \prod_{i=1}^{\ell} (n-1-i)^2}{\prod_{i=-\ell-1}^{\ell+1} (n+i)} \\
    &\le (\ell+1)(\ell+2)(2\ell+3) \\
    &=2\ell^3 + 9\ell^2 + 13\ell + 6.
\end{align*}
Our upper bound of \cref{onlytwocontributions} therefore reads
\begin{align*}
    \chi_{e_n}^2(s) &\le \sum_{\ell \in \{1, \cdots, n\} \text{ even}}(2\ell^3+\ell^2+2\ell+1)(\beta_{\ell/2})^{2s} + \sum_{\ell \in \{0, \cdots, n-2\} \text{ odd}} (2\ell^3 + 9\ell^2 + 13\ell + 6)(\beta_{(1+\ell)/2})^{2s} \\
    &\le \sum_{\ell \in 2\NN} (2\ell^3+\ell^2+2\ell+1)(\beta_{\ell/2})^{2s} + \sum_{\ell \in 2\NN - 1}(2\ell^3 + 9\ell^2 + 13\ell + 6)(\beta_{(1+\ell)/2})^{2s} \\
    &= \sum_{k \in \NN} (32k^3 + 16k^2 + 6k + 1)(\beta_k)^{2s}
\end{align*}
(last step by making the substitutions $\ell = 2k, \ell = 2k-1$ in the summations). But thanks to the binomial coefficient bound $\binom{2a}{a} < \frac{4^a}{\sqrt{\pi a}}$ (\cite{stanica}, Theorem 2.5), we have $\beta_k < \frac{1}{\pi k}$ for all $k$, and the largest nontrivial eigenvalue is $\beta_1 = \frac{1}{4}$. Thus we can use the crude bound 
\[
    \beta_k^{2s} \le \left(\frac{1}{4}\right)^{2(s-3)} \beta_k^6 < \frac{4^6}{\pi^6} \left(\frac{1}{4}\right)^{2s} \cdot \frac{1}{k^6}
\]
for all $s \ge 3$ so that the series converges. This yields
\begin{align*}
    \chi_{e_n}^2(s) &\le \frac{4^6}{\pi^6}\sum_{k \in \NN} \frac{32k^3 + 16k^2 + 6k + 1}{k^6} \left(\frac{1}{4}\right)^{2s} \\
    &\le \frac{4^6}{\pi^6} \cdot 63.1\cdot \left(\frac{1}{4}\right)^{2s}\\
    &\le 270 \left(\frac{1}{4}\right)^{2s},
\end{align*}
completing the proof.
\end{proof}

A similar analysis may be carried out from other starting states as well, though there will be even more nonzero terms in the expression for $\chi_x^2(s)$ for general states $x$. We leave these calculations as potential future work, though we believe that a similar strategy as what we have described here may be fruitful and that the asymptotics will be relatively well-behaved due to the nice form of the norms $\langle f_Q^{m,\ell}, f_Q^{m,\ell}\rangle_{\pi}$.

\section{Miscellaneous remarks}\label{miscremarkssection}

In this final section, we collect some additional facts about our matrix $K_n$ in this algebraic framework. One of the key properties of our diagonalization is \cite[Proposition 3.3]{diaconislinram1}, which states that the restriction of the binary Burnside process to any subset of its coordinates is exactly the Burnside process on that coordinate set. We first write out an algebraic variant of that, which (in words) says that ``lumping over the final coordinate still yields the Burnside process on the remaining coordinates.''

\begin{proposition}\label{Kprojalgebraic}
Let $I$ denote the identity $2 \times 2$ matrix. We have, for any $n \ge 2$,
\[
    K_n(I^{\otimes(n-1)}\otimes K_1) =  K_{n-1}\otimes K_1.
\]
\end{proposition}
\begin{proof}
We have the closed-form expression from \cite[Proposition 3.1]{diaconislinram1}
\begin{equation}\label{closedformK}
    K_n(x, y) = \frac{\binom{2n_{00}}{n_{00}} \binom{2n_{01}}{n_{01}} \binom{2n_{10}}{n_{10}} \binom{2n_{11}}{n_{11}}}{4^n\binom{n_{00} + n_{01}}{n_{00}}\binom{n_{10} + n_{11}}{n_{10}}}
\end{equation}
A direct computation shows that
\[
    \frac{\binom{2(n_{00}+1)}{n_{00}+1}}{\binom{n_{00}+n_{01}+1}{n_{00}+1}} = \frac{\frac{(2n_{00}+1)(2n_{00}+2)}{(n_{00}+1)(n_{00}+1)}} {\frac{(n_{00}+n_{01}+1)}{(n_{00}+1)}} \frac{\binom{2n_{00}}{n_{00}}}{\binom{n_{00}+n_{01}}{n_{00}}} = \frac{2(2n_{00}+1)}{(n_{00}+n_{01}+1)} \frac{\binom{2n_{00}}{n_{00}}}{\binom{n_{00}+n_{01}}{n_{00}}},
\]
as well as the same identity with $n_{00}$ and $n_{01}$ switched. Thus for any $x, y \in C_2^{n-1}$ (and defining $n_{00}, n_{01}, n_{10}, n_{11}$ relative to $x, y$), 
\begin{align*}
K_n(x0,y0)+K_n(x0,y1) &= \left(\frac{2(2n_{00}+1)}{n_{00}+n_{01}+1}+\frac{2(2n_{01}+1)}{n_{00}+n_{01}+1}\right) \cdot \frac{\binom{2n_{00}}{n_{00}} \binom{2n_{01}}{n_{01}} \binom{2n_{10}}{n_{10}} \binom{2n_{11}}{n_{11}}}{4^n\binom{n_{00} + n_{01}}{n_{00}}\binom{n_{10} + n_{11}}{n_{10}}} \\
&= \frac{\binom{2n_{00}}{n_{00}} \binom{2n_{01}}{n_{01}} \binom{2n_{10}}{n_{10}} \binom{2n_{11}}{n_{11}}}{4^{n-1}\binom{n_{00} + n_{01}}{n_{00}}\binom{n_{10} + n_{11}}{n_{10}}} \\
&= K_{n-1}(x,y),
\end{align*}
where we have used that $n_{00}$ increments from $(x, y)$ to $(x0, y0)$ and $n_{01}$ increments from $(x, y)$ to $(x0, y1)$, but all other values stay the same. Using the same strategy, we also have that $K_n(x1,y0) + K_n(x1,y1) = K_{n-1}(x,y)$ for all $x, y$. Putting this together and using that $K_1 = \begin{bmatrix} 1/2 & 1/2 \\ 1/2 & 1/2 \end{bmatrix}$, we arrive at the desired equality
\[
    K_n(I^{\otimes(n-1)}\otimes K_1) = K_{n-1}\otimes K_1
\]
(in words, this says that we can either average over the last coordinate before or after applying the binary Burnside matrix).
\end{proof}

Next, we show (analogously to the proof of \cref{eigenvectors} appearing in \cite{diaconislinram1}) an algebraic proof that the ``lifted vectors'' $f_S = \sum_{T\subseteq \{1, \ldots, n\}} (-1)^{\vert S\cap T\vert} \binom{\vert S\vert}{\vert S\cap T\vert} v_T$ are indeed eigenvectors of our Burnside matrix.

\begin{proposition}
Let $S\subseteq \{1, \ldots, n\}$ and let $\ell = |S|$.  Then $f_S$ is an eigenvector of $K_n$ of eigenvalue $\beta_{\ell/2}$ if $\ell$ is even and $0$ otherwise.
\end{proposition}
\begin{proof}
The case $S = \varnothing$ corresponds to the vector $f_S = \sum_{T \subseteq \{1, \cdots, n\}} v_T$, which is indeed an eigenvector of the row-stochastic matrix $K_n$ with eigenvalue $\beta_0 = 1$. From here on, we consider $\ell > 0$. 

First, assume that $S = \{1, \cdots, \ell\}$ (meaning $v_S = v_1^{\otimes \ell} \otimes v_0^{\otimes (n-\ell)}$). Write 
\[
    I = \begin{bmatrix} 1 &0 \\ 0 &1 \end{bmatrix} \quad\hbox{and}\quad K_1 = \frac{1}{2}\begin{bmatrix} 1 & 1 \\ 1 &1 \end{bmatrix}.
\]
Applying $K_1$ to each of the last $(n-\ell)$ coordinates yields
\[
    K_1^{\otimes(n-\ell)}v_{0^{(n-\ell)}}  = \frac{1}{2^{n-\ell}} (v_0+v_1)^{\otimes(n-\ell)} = \frac{1}{2^{n-\ell}} \sum_{T \subseteq \{1, \ldots, n-\ell\}} v_T.
\]
Therefore, we can write 
\[
    f_S = f_{\{1, \cdots, \ell\}} = 2^{n-\ell}(I^{\otimes \ell}\otimes K_1^{\otimes (n-\ell)}) (f_{\{1, \cdots, \ell\}}\otimes v_{0}^{\otimes (n-\ell)}).
\]
Hence letting $\beta$ denote the corresponding eigenvalue (either $\beta_{\ell/2}$ if $\ell$ is even or $0$ otherwise), we have
\begin{align*}
K_n f_{\{1, \cdots, \ell\}} 
&= 2^{n-\ell} K_n(I^{\otimes \ell}\otimes K_1^{\otimes (n-\ell)})(f_{\{1, \cdots, \ell\}}\otimes v_{0}^{\otimes (n-\ell)}) \\
&= 2^{n-\ell}(K_\ell \otimes K_1^{\otimes (n-\ell)})(f_{\{1, \cdots, \ell\}}\otimes v_{0}^{\otimes (n-\ell)} )  \\
&= 2^{n-\ell} \beta (I^{\otimes \ell} \otimes K_1^{\otimes (n-\ell)} ) (f_{\{1, \cdots, \ell\}}\otimes v_{0}^{\otimes (n-\ell)})\\
&= \beta f_{\{1, \cdots, \ell\}} ,
\end{align*}
where we used \cref{Kprojalgebraic} in the second line and that $f_{\{1, \cdots, \ell\}}$ is an eigenfunction of $K_\ell$ in the third line. This proves that $f_{\{1, \cdots, \ell\}}$ is indeed an eigenvector of the correct eigenvalue. 

Finally, for the general case, for any $S\subseteq \{1, \ldots, n\}$ of size $\ell$, there exists $w \in S_n$ such that $f_S = w f_{\{1, \cdots, \ell\}}$. Therefore
\[
    K_n f_S = K_n w f_{\{1, \cdots, \ell\}} =  w K_n f_{\{1, \cdots, \ell\}} = \beta w f_{\{1, \cdots, \ell\}} = \beta f_S,
\]
as desired.
\end{proof}

We now describe one more curious property of the matrices $K_n$. As previously discussed, $K_n$ commuting with the action of $S_n$ implies that it may be viewed as an element of the universal enveloping algebra $U(\mathfrak{sl}_2)$. Thus, $K_n$ may be rewritten as some polynomial expression in the basis elements $e, f, h$ of $\mathfrak{sl}_2$ as given in \cref{efhdefn}. In particular, the fact that the nonzero eigenvalues obtained in \cref{eigenvectors} do not depend on $n$ suggests that there may be a single element of $U(\mathfrak{sl}_2)$ which agrees with $K_n$ as an operator on $V^{\otimes n}$, or alternatively that there is some recursive relation among the $K_n$s which makes \cref{Kprojalgebraic} more transparent. The following result is one direction in which this idea could be further explored:

\begin{proposition}\label{pplusexpression}
Define the $2 \times 2$ matrices
\[
    p^+ = \frac{1}{2}(1+e+f) = \begin{bmatrix} 1/2 & 1/2 \\ 1/2 & 1/2 \end{bmatrix}, \quad p^- = \frac{1}{2}(1 - e - f) = \begin{bmatrix} 1/2 & -1/2 \\ -1/2 & 1/2 \end{bmatrix},  \quad p^+h = \frac{1}{2}(1+e+f)h = \begin{bmatrix} 1/2 & -1/2 \\ 1/2 & -1/2 \end{bmatrix}.
\]
Let $f(x, y, z)$ be the sum over all ways (orders) of taking the matrix tensor product of $x$ copies of $p^+$, $y$ copies of $p^-$, and $z$ copies of $p^+h$. Then
\[
    K_n = \sum_{x+y+z=n} c_{y,z} f(x, y, z), \quad \text{where} \quad c_{y, z} = \begin{cases}\left(\frac{y!z!}{\left(\frac{y}{2}\right)!\left(\frac{z}{2}\right)!\left(\frac{y+z}{2}\right)!2^{y+z}}\right)^2, & \text{ if }y, z \text{ are nonnegative even integers}, \\ 0, & \text{otherwise}. \end{cases}
\]
\end{proposition}

For example, the expressions for $n = 4$ and $n = 6$ read
\[
    K_4 = f(4, 0, 0) + \frac{1}{4}(f(2,2,0) + f(2,0,2)) + \frac{9}{64} (f(0,4,0) + f(0,0,4)) + \frac{1}{64} f(0, 2, 2),
\]  
\begin{align*}
    K_6 = f(6, 0, 0) + \frac{1}{4}(f(4,2,0) + f(4,0,2)) + \frac{9}{64} (f(2,4,0) + f(2,0,4)) + \frac{1}{64} f(2, 2, 2) \\
    + \frac{25}{256} (f(0,6,0) + f(0,0,6)) + \frac{1}{256}(f(0,2,4) + f(0,4,2)).
\end{align*}
Also note that the constants $c_{2k,0} = c_{0, 2k}$ are exactly the eigenvalues $\beta_k$ of our Markov chain. It may be interesting to obtain more explicit expressions for the various terms $f(x, y, z)$ -- observe in particular that $f(n, 0, 0)$ is always $\frac{1}{2^n}$ times the all-ones matrix, while all other terms have all row sums equal to zero.

\cref{pplusexpression} was first independently proved by Ferenc Balogh by rewriting the matrix $K_n$ in terms of tensor products. We present an argument based on that proof here. Since the ideas are rather different from the rest of the paper, we skip some of the minor computations for brevity.

\begin{proof}[Proof of \cref{pplusexpression}]
First, we may rewrite \cref{closedformK} in terms of beta functions $B(z_1, z_2) = \frac{\Gamma(z_1)\Gamma(z_2)}{\Gamma(z_1 + z_2)}$ to get the alternative expression
\[
    K(x, y) = \frac{B(n_{00} + \frac{1}{2}, n_{01} + \frac{1}{2})}{B(\frac{1}{2}, \frac{1}{2})} \cdot \frac{B(n_{10} + \frac{1}{2}, n_{11} + \frac{1}{2})}{B(\frac{1}{2}, \frac{1}{2})}.
\]
Now plugging in the integral representation of the beta function $B(z_1, z_2) = \int_0^1 t^{z_1 - 1} (1-t)^{z_2 - 1} dt$ yields 
\begin{align}\label{mixtureformula}
    K(x, y) &= \frac{1}{\pi^2} \int_{[0, 1]} \int_{[0, 1]} \frac{s^{n_{00}} (1-s)^{n_{01}}}{\sqrt{s(1-s)}} \frac{t^{n_{10}} (1-t)^{n_{11}}}{\sqrt{t(1-t)}} ds dt \nonumber \\ 
    &= \frac{1}{(2\pi)^2} \int_0^{2\pi} \int_0^{2\pi} \cos^{2n_{00}}(\theta_1) \sin^{2n_{01}}(\theta_1) \cos^{2n_{10}}(\theta_2) \sin^{2n_{11}}(\theta_2) d\theta_1 d\theta_2 \nonumber \\
    &= \frac{1}{(2\pi)^2} \int_0^{2\pi} \int_0^{2\pi} \begin{bmatrix} \cos^2(\theta_1) & \sin^2(\theta_1) \\ \sin^2(\theta_2) & \cos^2(\theta_2) \end{bmatrix}^{\otimes n} d\theta_1 d\theta_2
\end{align}
using the change of variables $s = \cos^2(\theta_1), t = \cos^2(\theta_2)$. With a further change of variables $\theta_1 = \frac{\alpha - \beta}{2}$, $\theta_2 = \frac{\alpha + \beta}{2}$, we have
\[
    \begin{bmatrix} \cos^2(\theta_1) & \sin^2(\theta_1) \\ \sin^2(\theta_2) & \cos^2(\theta_2) \end{bmatrix} = p^+ + \cos(\alpha) \cos(\beta) p^- + \sin(\alpha) \sin(\beta) p^+h.
\]
Therefore, if we expand out the $n$-fold tensor product of this quantity and group based on the number of terms of each type, we obtain (after some computation for the change-of-variables)
\begin{align*}
    K_n = \sum_{x+y+z=n} c_{y,z} f(x, y, z), \quad \text{where } c_{y, z} &= \frac{1}{(2\pi)^2} \int_0^{2\pi} \int_0^{2\pi} (\cos \alpha \cos \beta)^y (\sin \alpha \sin \beta)^z d\alpha d\beta \\
    &= \left(\frac{1}{2\pi}\int_0^{2\pi} \cos^y \alpha \sin^z \alpha\, d\alpha \right)^2.
\end{align*}
Now if either of $y$ or $z$ are odd, the integral is zero by symmetry; if both are even, then by the trigonometric form of the beta function we have
\begin{align*}
    c_{y, z} = \,&\left(\frac{1}{\pi}B\left(\frac{y+1}{2}, \frac{z+1}{2}\right)\right)^2 = \left(\frac{\Gamma\left(\frac{y+1}{2}\right)\Gamma\left(\frac{z+1}{2}\right)}{\pi \Gamma(\frac{y}{2} + \frac{z}{2} + 1)} \right)^2 \\
    &= \left(\frac{y!}{2^y \left(\frac{y}{2}\right)!} \frac{z!}{2^z \left(\frac{z}{2}\right)!}\frac{1}{\left(\frac{y+z}{2}\right)!} \right)^2,
\end{align*}
which recovers the desired constants.
\end{proof}

The observation \cref{mixtureformula} that the matrix $K_n$ can be written as a mixture of $n$-tensored $2 \times 2$ matrices is part of a more general phenomenon. A sequence of Markov chains $K_n$ which satisfy \cref{Kprojalgebraic} and which commute with the action of $S_n$, which we call ``super-lumped chains,'' can be characterized with the help of de Finetti's notion of \textit{partial exchangeability}. This characterization in fact allows for a basis analogous to $\{f_Q^{m, \ell}\}$ to be constructed for any sequence of super-lumped chains on the hypercube. Forthcoming work in this direction will elaborate on this connection and its consequences.

\medskip

Finally, we comment on the potential for extending this algebraic work beyond the binary case; more probabilistic remarks on this generalization can be found in \cite[Section 6.2]{diaconislinram1}. In one step of this more general Burnside process on $(C_k^n, S_n)$, we begin with an $n$-tuple $x \in C_k^n$, uniformly pick a permutation permuting the coordinates within each value, write that permutation as a product of disjoint cycles, and label each cycle uniformly with one of the $k$ values in the alphabet. Our goal would be to also arrive at an orthogonal basis of eigenvectors for the $k^n$ by $k^n$ transition matrix of that Markov chain.

In place of the decomposition in \cref{irreddecomp}, we must now consider permutation representations $M^\lambda$ for partitions $\lambda$ of $n$ of up to $k$ parts, and we now have
\[
    M^{\lambda} = \bigoplus_{\mu} K_{\mu \lambda} S^{\mu}
\]
where $K_{\mu\lambda}$ are the \textit{Kostka numbers} (which are positive if and only if $\mu$ dominates $\lambda$). Towards understanding how these (many) copies of $S^{\mu}$ are arranged in the function space, the Schur--Weyl decomposition of \cref{schurweyldecomp} now reads
\[
    L(C_k^n) = \bigoplus_{\substack{\lambda \text{ partition of }n \\ \text{with at most }k\text{ parts}}} S^\lambda \otimes L^\lambda,
\]
where the $S^\lambda$s are again Specht modules and the $L^\lambda$s are now irreducible representations of $\mathfrak{sl}_k$. One natural question is whether there is an explicit expression for the transition matrix as an element of the universal enveloping algebra $U(\mathfrak{sl}_k)$ in the same way as \cref{pplusexpression}, and whether that expression can be written in a way that demonstrates how $K_n$s relate to each other for various $n$.

We note that we do not even have a full description of the eigenvalues for the general $(C_k^n, S_n)$ Burnside process, though we do have the following conjecture:

\begin{conjecture}[{\cite[Conjecture 6.2]{diaconislinram1}}]\label{ckmultiplicityconjecture}
Fix $k$, and let $\lambda$ be any nonzero eigenvalue of the Burnside chain on $(C_k^n, S_n)$ for any $n$. Then $\lambda$ occurs with multiplicity $a_\lambda \binom{n}{b_\lambda}$ for some integers $a_\lambda, b_\lambda$.
\end{conjecture}

This conjecture may be resolved by proving some appropriate generalization of \cref{eigenvaluesplittingresult}; however, the expressions for $\dim(S^\lambda)$ (and thus the contributions to the total eigenvalue multiplicities) are in general the number of standard Young tableaux of shape $\lambda$, which may be more cumbersome to work with than the simpler expressions $\dim(S^{(n-i, i)}) = \binom{n}{i} - \binom{n}{i-1}$ that appear in the $k = 2$ case.

\appendix

\section{Verifying orthogonality of the remaining eigenvectors}

The formula for the inner product $\langle f_T^{m, \ell_1}, f_T^{m, \ell_2} \rangle$ can be evaluated using the ``creative telescoping'' method with computer assistance. This is the same strategy used to show \cite[Proposition 4.3] {diaconislinram1}, so reading the proof of that result may also be insightful. The key ideas of this algorithm come from Wilf and Zeilberger's WZ method -- an overview can be found in \cite{wzmonthly} -- and subsequent work has been done to speed up the algorithm with various heuristics and a careful ansatz \cite{koutschan}. The Mathematica package \texttt{HolonomicFunctions} that we used, along with further literature references, may be found at \url{https://risc.jku.at/sw/holonomicfunctions/}. 

AI tools were used to obtain some of the technical ideas in the argument below, but the proof here is written fully by the authors. We remark that this lemma can also be proved without using creative telescoping computations, specifically by using properties of the ``mixture of tensors'' representation in the proof of \cref{pplusexpression} in a more general setting. More details on this alternate approach will appear in future work.

\begin{lemma}\label{tableauorthogonality}
Using the notation in \cref{schurweylsection}, in particular \cref{lperp}, \cref{Tscalarsdefn}, and \cref{fgTd}, we have
\[  
    \langle f_T^{m, \ell_1}, f_T^{m, \ell_2} \rangle = 0
\]
for any integer $m \in \{0, 1, \ldots, \lfloor n/2 \rfloor\}$ and $\ell_1 \ne \ell_2 \in \{0, 1, \ldots, n -2m\}$, where $T$ is the column reading tableau of shape $(n-m, m)$. 
\end{lemma}
\begin{proof}
First, we unpack the definitions to write the left-hand side as an explicit sum. We have
\begin{align*}
    \langle f_T^{m, \ell_1}, f_T^{m, \ell_2} \rangle &= \sum_{i, j=0}^{n-2m} T_{m,n}^{(\ell_1)}(i) T_{m,n}^{(\ell_2)}(j) \langle g_T^{m, i}, g_T^{m, j} \rangle \\
    &= \frac{2^m}{n+1}\sum_{i=0}^{n-2m} T_{m,n}^{(\ell_1)}(i) T_{m,n}^{(\ell_2)}(i) \frac{\binom{n-2m}{i}}{\binom{n}{m+i}}
\end{align*}
by the orthogonality of the $g_T^{m, i}$s and the computation \cref{gTnorm}. Writing out the $T_{m,n}^{(\ell)}(i)$s using \cref{formulafortscalars}, we wish to prove that
\[
    \sum_{i=0}^{n-2m} \sum_{j_1=0}^i \sum_{j_2=0}^i  (-1)^{j_1 + j_2} \binom{2m + \ell_1}{m+j_1} \binom{2m + \ell_2}{m+j_2} \binom{i}{j_1} \binom{i}{j_2}\binom{n-2m-i}{\ell_1-j_1}\binom{n-2m-i}{\ell_2-j_2}\frac{\binom{n-2m}{i}}{\binom{n}{m+i}} \stackrel{?}{=} 0
\]
whenever $\ell_1 \ne \ell_2$. Noting that $\frac{\binom{n-2m}{i}}{\binom{n}{m+i}} = \frac{1}{\binom{n}{n-2m, m, m}}\binom{m+i}{i}\binom{n-m-i}{m}$, it suffices to prove that for all nonnegative integers $n, m, \ell_1, \ell_2$ with $n \ge 2m$ and $\ell_1 \ne \ell_2$, we have
\[
    \scalebox{.8}{$\displaystyle\sum_{i=0}^{n-2m} \sum_{j_1=0}^i \sum_{j_2=0}^i  (-1)^{j_1 + j_2} \binom{2m + \ell_1}{m+j_1} \binom{2m + \ell_2}{m+j_2} \binom{i}{j_1} \binom{i}{j_2}\binom{n-2m-i}{\ell_1-j_1}\binom{n-2m-i}{\ell_2-j_2}\binom{m+i}{i} \binom{n-m-i}{m} \stackrel{?}{=} 0$}.
\]
Define $P = P(n, m, \ell_1, \ell_2, i, j_1, j_2)$ to be this summand, except with binomial coefficients $\binom{a}{b}$ replaced with products of gamma and reciprocal gamma functions $\Gamma(a+1)\frac{1}{\Gamma(b+1)} \frac{1}{\Gamma(a-b+1)}$ and $(-1)^{j_1 + j_2}$ replaced with $e^{\mathbf{i} \pi(j_1 + j_2)}$, and also define the functions $Q_I, Q_{J_1}, Q_{J_2}$ via
\[
    Q_I = P \cdot \frac{(j_2 - j_1)(n - m - i + 1)(n - 2m - i + 1)}{(\ell_1 - \ell_2)(n - 2m - i + j_1 - \ell_1 + 1)(n - 2m - i + j_2 - \ell_2 + 1)},
\]
\[
    Q_{J_1} = P \cdot \frac{-j_1(j_1 + m)}{(1+i-j_1)(\ell_1-\ell_2)},
\]
\[
    Q_{J_2} = P \cdot \frac{j_2(j_2+m)}{(1+i-j_2)(\ell_1-\ell_2)}.
\]
To avoid poles at nonpositive integers for the gamma functions, let $\varepsilon = (c_it, c_1t, c_2t)$ for $c_i, c_1, c_2$ distinct values in $(0, 1)$, and define $P_\varepsilon = P(n, m, \ell_1, \ell_2, i + c_it, j_1 + c_1t, j_2 + c_2t)$. Similarly define $Q_{I,\varepsilon}$, $Q_{J_1,\varepsilon}$, and $Q_{J_2, \varepsilon}$ by using $P_\varepsilon, i + c_it, j_1 + c_1t, j_2 + c_2t$ in place of $P, i, j_1, j_2$. Then the following identity holds (as an equality of meromorphic functions):
\begin{equation}\label{telescopingeqn}
    P_{\varepsilon}(n+1) - P_{\varepsilon}(n) = Q_{I,\varepsilon}(i+1) - Q_{I,\varepsilon}(i) + Q_{J_1,\varepsilon}(j_1+1) - Q_{J_1,\varepsilon}(j_1) + Q_{J_2,\varepsilon}(j_2 + 1) - Q_{J_2,\varepsilon}(j_2).
\end{equation}
Indeed, this identity can be verified by dividing through by $P_\varepsilon(n)$, at which point all terms become rational functions in the variables and clearing denominators yields equal polynomials on both sides. If we now sum both sides of \cref{telescopingeqn} over $(i, j_1, j_2) \in \{0, 1, \ldots, n - 2m \} \times \{0, 1, \ldots, N\}^2$ for a positive integer $N > \max(n + 1, \ell_1, \ell_2)$, we find  
\begin{align*}
    \sum_{i=0}^{n-2m}\sum_{j_1, j_2 = 0}^N P_\varepsilon(n+1) - \sum_{i=0}^{n-2m}\sum_{j_1, j_2 = 0}^N P_\varepsilon(n) = \sum_{j_1, j_2 = 0}^N &\left(Q_{I, \varepsilon}(n-2m+1) - Q_{I, \varepsilon}(0)\right) \\
    &+ \sum_{i=0}^{n-2m}\sum_{j_2 = 0}^N \left(Q_{J_1, \varepsilon}(N+1) - Q_{J_1, \varepsilon}(0)\right)  \\
    &+ \sum_{i=0}^{n-2m}\sum_{j_1 = 0}^N \left(Q_{J_2, \varepsilon}(N+1) - Q_{J_2, \varepsilon}(0)\right).
\end{align*}
Now send $t \to 0$ on both sides (avoiding poles by our choice of $c_i, c_1, c_2$). For this range of $(i, j_1, j_2)$ values, the upper arguments of all binomial coefficients approach nonnegative integers, so the gamma function representations are valid when taking limits. Consider the right-hand side:
\begin{itemize}
\item For the latter two sums, observe that $Q_{J_1, \varepsilon}(N+1), Q_{J_2, \varepsilon}(N+1) \to 0$ because the denominators of $Q_{J_1}$ and $Q_{J_2}$ are nonzero and the gamma function terms corresponding to $\binom{i}{j_1}, \binom{i}{j_2}$ tend to zero. Similarly, $Q_{J_1,\varepsilon}(0), Q_{J_2,\varepsilon}(0)$ tend to zero because of the factors of $j_1, j_2$ in the numerator, respectively. Thus the contribution from these sums vanish in the limit.
\item For the former sum, first observe that
\[
    Q_I(n, i, j_1, j_2) = \frac{j_2 - j_1}{\ell_1 - \ell_2} P(n+1, i, j_1, j_2)
\]
(by dividing both sides by $P(n, i, j_1, j_2)$ and comparing the resulting rational functions), and so the same identity is also true for $Q_{I,\varepsilon}$ and $P_{\varepsilon}$ but with the perturbed values of $j_1, j_2$. Thus we can instead consider the limit 
\[
    \lim_{t \to 0} \sum_{j_1, j_2 = 0}^N \frac{j_2 - j_1 + (c_2 - c_1)t}{\ell_1 - \ell_2} \left(P_{\varepsilon}(n+1, n-2m+1, j_1, j_2) - P_{\varepsilon}(n+1, 0, j_1, j_2)\right).
\]
We always have $\frac{j_2 - j_1 + (c_2 - c_1)t}{\ell_1 - \ell_2} P_\varepsilon(n+1, 0, j_1, j_2) \to 0$ because either the gamma function terms for $\binom{i}{j_1} \binom{i}{j_2}$ tend to zero, or $j_1 = j_2 = 0$ so the numerator $(j_2 - j_1 + (c_2 - c_1)t)$ tends to zero. Similarly, $P_\varepsilon(n+1, n-2m+1, j_1, j_2) \to 0$ unless $j_1 = \ell_1$ and $j_2 = \ell_2$, meaning the only nonzero contribution as $t \to 0$ is
\[
    \frac{\ell_2 - \ell_1}{\ell_1 - \ell_2} P(n+1, n-2m + 1, \ell_1, \ell_2) = -P(n+1, n-2m+1, \ell_1, \ell_2)
\]
(here is where we use that $N > \ell_1, \ell_2$).
\end{itemize}

Meanwhile, on the left-hand side we have $P_\varepsilon \to P$ for all terms. Thus,
\[
    \sum_{i=0}^{n-2m} \sum_{j_1=0}^i \sum_{j_2=0}^i P(n+1, i, j_1, j_2) - \sum_{i=0}^{n-2m} \sum_{j_1=0}^i \sum_{j_2=0}^i P(n, i, j_1, j_2) = -P(n+1, n-2m+1, \ell_1, \ell_2).
\]
But this in fact implies that
\[
    \sum_{i=0}^{(n+1)-2m} \sum_{j_1=0}^i \sum_{j_2=0}^i P(n+ 1, i, j_1, j_2) - \sum_{i=0}^{n-2m} \sum_{j_1=0}^i \sum_{j_2=0}^i P(n, i, j_1, j_2) = 0,
\]
since the ``missing'' $i = n + 1 - 2m$ part of the summation on the left-hand side only has a nonzero contribution if $j_1 = \ell_1$ and $j_2 = \ell_2$ because of the $\binom{(n+1)-2m-i}{\ell_1 - j_1}$ and $\binom{(n+1)-2m-i}{\ell_2 - j_2}$ terms. Thus, our sum of interest is independent of $n$. 

To evaluate it, plug in $n = 2m$, so that the only valid term in the summation is $i = 0, j_1 = 0, j_2 = 0$. Then $\binom{n-2m-i}{\ell_1 - j_1} \binom{n-2m-i}{\ell_2 - j_2}$ must be zero because at least one of $\ell_1$ and $\ell_2$ is positive by assumption, meaning the entire summation is indeed zero. Thus we've proven our desired identity, concluding the proof.
\end{proof}

\bibliographystyle{bibstyle}
\bibliography{biblio}

\end{document}